\documentclass[11pt]{amsart}
\usepackage{amsmath}
\usepackage{amssymb}
\usepackage{amsfonts}
\usepackage{epsfig}

\usepackage[all]{xy}
\newcommand{\ri}{\rightarrow}

\newcommand{\Ga}{\Gamma}
\newcommand{\usr}{\overset\sim\rightarrow}
\newcommand{\p}{\pi_1}
\newcommand{\CP}{\mathbb{CP}}
\newcommand{\Int}{\rm Int}
\newcommand{\h}{h}

\newcommand{\al}{\alpha}
\newcommand{\vp}{\varphi}

\newcommand{\N}{\mathbb{N}}
\newcommand{\C}{\mathbb{C}}
\newcommand{\CC}{\mathbb{C}}
\newcommand{\ZZ}{\mathbb{Z}}
\newcommand{\PP}{\mathbb{P}}
\newcommand{\FF}{\mathbb{F}}
\newcommand{\RR}{\mathbb{R}}
\newcommand{\Z}{\mathbb{Z}}
\newcommand{\G}{\Gamma}
\newcommand{\A}{\mathcal{A}}

\newcommand{\LL}{\mathcal{L}}

\newtheorem{thm}{Theorem}[section]
\newtheorem*{thm*}{Theorem}
\newtheorem{corollary}[thm]{Corollary}
\newtheorem{lemma}[thm]{Lemma}
\newtheorem{step}[thm]{Step}
\newtheorem{conjecture}[thm]{Conjecture}
\newtheorem*{lemma*}{Lemma}
\newtheorem{definition}[thm]{Definition}
\newtheorem{example}[thm]{Example}
\newtheorem{prs}[thm]{Proposition}
\newtheorem{remark}[thm]{Remark}
\newtheorem{assumption}[thm]{Assumption}
\newtheorem{notation}[thm]{Notation}

\newtheorem*{notation*}{Notation}

\begin{document}

\title [Structure of CF fundamental groups and CL arrangements]{On the structure of conjugation--free fundamental groups of conic--line arrangements}

\author{Michael Friedman and David Garber}

\address{Michael Friedman, Institut Fourier, 100 rue des maths, BP 74, 38402 St Martin d'H\'eres cedex, France; Max Planck Institute for Mathematics, Vivatsgasse 7, 53111 Bonn, Germany}
\email{Michael.Friedman@ujf-grenoble.frþ}

\address{David Garber, Department of Applied Mathematics, Faculty of
  Sciences, Holon Institute of Technology, 52 Golomb st., PO
  Box 305, 58102 Holon, Israel}
\email{garber@hit.ac.il}

\begin{abstract}

The fundamental group of the complement of a hyperplane arrangement plays an important role in studying the corresponding arrangements.
In particular, for large families of hyperplane arrangements, this fundamental group, being isomorphic to the fundamental group of a complement of a line arrangement, has some remarkable properties: either it is a direct sum of free groups and a free abelian group, or it has a conjugation-free geometric presentation.

In this paper, we first give a complete proof to the following key lemma: if we draw a new line through only one intersection point of a given real line arrangement whose fundamental group is conjugation-free, then the fundamental group of the new arrangement is also conjugation-free.

 Second, we generalize this lemma to the case of conic-line arrangements. Moreover, we prove that once the graph associated to conic-line arrangements (defined slightly different than the corresponding graph for line arrangements) has no cycles, then the fundamental group of its complement has a conjugation-free geometric presentation and in addition can be written as a direct sum of free groups and a free abelian group. Also, we show that if the graph consists of one cycle, and the conic does not pass through all the multiple points  corresponding to the vertices of the cycle, then the fundamental group has a conjugation-free geometric presentation as well.

 For conclusion, we extend the family of real line arrangements having a conjugation-free geometric presentation (for their fundamental group)
by defining the notion of a conjugation-free graph. We also extend this notion to certain families of conic-line arrangements.

\end{abstract}

\maketitle


\section{Introduction}

The fundamental group of the complement of a plane curve is a very important topological invariant.
For example, it is used to distinguish between curves that form a Zariski pair, which is a pair of curves having the same combinatorics but
non-homeomorphic complements in $\CC\PP^2$ (see \cite{AB-CR} for the exact definition and \cite{ABCT} for a survey).
Another example is that while the fundamental group of the complement of a nodal curve is abelian (see \cite{Za2}), there are curves with non-abelian fundamental groups. Thus,
it is interesting to explore finite non-abelian groups which
arise that way,
see for example \cite{AB,AB1,Deg,Z1}.

Moreover, the Zariski-Lefschetz hyperplane section theorem (see \cite{milnor})
states that
$\pi_1 (\CC\PP^N - S) \cong \pi_1 (H - (H \cap S)),$
where $S$ is a hypersurface and $H$ is a generic 2-plane.
Since $H \cap S$ is a plane curve, the fundamental groups of complements of plane curves
can also be used for computing the fundamental groups of complements of hypersurfaces  in $\CC\PP^N$.
Note that when $S$ is a hyperplane arrangement, $H \cap S$ is a line arrangement in $\CC\PP^2$. Thus, one of the main tools for investigating the topology of hyperplane arrangements is the fundamental groups $\pi_1(\CC\PP^2- \LL)$ and $\pi_1(\CC^2 - \LL)$, where $\LL$ is a line arrangement.

For line arrangements these groups have very interesting properties (see e.g. \cite[Section 5.3]{OT}). They are abelian if and only if $\LL$ has only nodes as intersection points, see for example \cite[Example 1.6(a)]{DOZ}. Moreover, Fan \cite{Fa2} and Eliyahu et al. \cite{ELST} proved that this group is a direct sum of a free abelian group and free groups if and only if a certain graph, associated to the intersection points of $\LL$ (see Section \ref{secLineArr}), has no cycles. Based on this, Eliyahu et al. \cite{EGT1,EGT2} showed that other properties hold for certain presentations of this group: conjugation-free and complemented presentations.

\medskip

As conic--line arrangements are a natural generalization of line arrangements, an immediate question that arises is whether the above properties (e.g. conjugation--freeness or a structure of a direct sum of a free abelian group and free groups) hold for these arrangements too. One should note, contrary to the situation for line arrangements, that this group can be abelian even if the conic--line arrangement has singular points which are not nodes (see \cite{Deg} and Figure \ref{pi1Ab} below).

Note that the fundamental groups $\pi_1(\CC\PP^2 - \A)$ and $\pi_1(\CC^2 - \A)$, for some families of conic-line arrangements $\A$, were studied by Amram et al. (see e.g. \cite{AT3,AT2} and especially \cite[Theorem 6]{AT1}). Moreover, Zariski pairs consisting of conic-line arrangements were studied by, for example, Namba-Tsuchihashi \cite{NT} and Tokunaga \cite{Tok}, but a research in the spirit of the above questions has not been carried out yet.

\medskip

In this paper, we generalize Fan's result to the case of conic--line arrangements. After surveying in Section \ref{secLineArr} the known results on line arrangements, the braid monodromy technique and the conjugation-free property, Section \ref{secRealLineArr} deals with the preservation of this property under certain actions for line arrangements, which corrects and completes the proofs given in \cite{EGT2}. Section \ref{secRealCLArr} examines this property for CL arrangements. In Section \ref{secNoCycles}, we present a necessary condition which implies that the fundamental group of some families of conic-line arrangements is a direct sum of a free abelian group and free groups. Explicitly, we generalize Fan's concept of a graph associated to line arrangements to the case of real conic-line arrangements and prove that once this graph has no cycles, then the corresponding fundamental group has the desired structure (for an explicit formulation, see Theorem \ref{main_result}). In Section \ref{secOneCycleNotPass}, we prove a few  propositions about the structure of the fundamental group of a conic-line arrangement whose graph consists of a single cycle, where the conic does not pass through all the multiple points corresponding to the vertices of the cycle.

In Section \ref{secCFGraphs}, we define the notion of a {\it conjugation-free graph}, for both  line arrangements and  conic-line arrangements. We show that for every arrangement whose graph is a conjugation-free graph, the fundamental group of its complement has a
conjugation-free  geometric presentation.
\medskip

{\textbf{Acknowledgements}}: We would like to thank Meital Eliyahu for stimulating talks. Also, we would like to thank an anonymous referee of a previous version of this paper for useful and important suggestions, especially regarding the proof of the preservation of the conjugation-free property.

 The first author would like to thank the Max-Planck-Institute f\"ur Mathematik in Bonn for the warm hospitality and support and the Fourier Institut in Grenoble, where the final part of this paper was carried out.

\section{Arrangements, Braid monodromy and Conjugation-Free property} \label{secLineArr}

In this section, we give a short survey of some known results concerning the structure of the fundamental group of the complement of a line arrangement. After that, we present the family of conic--line arrangements that we deal with and give a short survey about the braid monodromy technique, for computing presentations of fundamental groups of complements of plane curves. In the last subsection, we present the notion of conjugation-free geometric presentation of a fundamental group associated to a line arrangement or to a conic-line arrangement.

\subsection{Line arrangements and conic-line arrangements}

An {\it affine line arrangement} in $\CC^2$ is a union of copies of $\C^1$ in $\C^2$. Such an arrangement is called {\em real} if the defining equations of all its lines can be written with real coefficients, and {\em complex} otherwise.

For real and complex line arrangements $\mathcal L$, Fan \cite{Fa2} defined a graph $G(\mathcal L)$ which is associated to its multiple points (i.e. points where more than two lines are intersected). We give here its version for real arrangements (the general version is more delicate to explain and will be omitted): Given a real line arrangement $\mathcal L$, the graph $G(\mathcal L)$ of multiple points lies on the real part of $\mathcal L$. It consists of the multiple points of $\mathcal L$ as vertices, with the
segments between the multiple points on lines which have at least two multiple points as edges. Note that if the arrangement consists of three
multiple points on the same line, then $G(\mathcal L)$ has three vertices on the same edge (see Figure \ref{graph_GL}(a)).
If two such lines happen to intersect in a simple point (i.e. a point where exactly two lines are intersected), it is ignored
(i.e. there is no corresponding vertex in the graph). See another example in Figure \ref{graph_GL}(b) (note that Fan's definition gives a graph slightly different from the graph defined in \cite{JY,WY}).

\begin{figure}[!ht]
\epsfysize 4cm
\centerline{\epsfbox{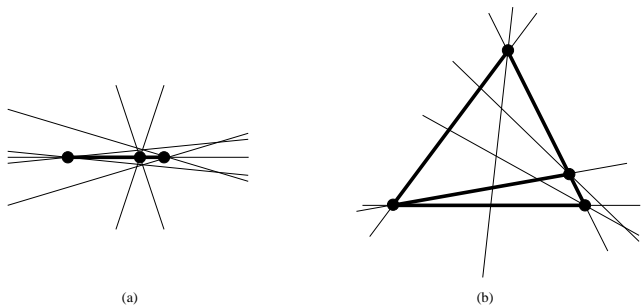}}
\caption{Examples for the graph $G(\mathcal L)$.}\label{graph_GL}
\end{figure}

\medskip

Fan \cite{Fa1,Fa2} proved the following result:

\begin{prs}[Fan]\label{Fan}
Let $\mathcal L$ be a complex arrangement of $k$ lines and $S=\{a_1, \dots, a_p\} $
be the set of all multiple points of $\mathcal L$.  Suppose that
$\beta (\mathcal L)=0$, where $\beta (\mathcal L)$ is the first Betti number of the graph $G(\mathcal L)$ (hence $\beta (\mathcal L)=0$
means that the graph $G(\mathcal L)$ has no cycles). Then:
$$\pi_1 (\CC  ^2 - \mathcal L) \cong \ZZ^r \oplus \bigoplus_{i=1}^p \FF_{m(a_i)-1},$$
where $m(a_i)$ is the multiplicity of the intersection point $a_i$ and\break $r=k+p-\sum\limits_{i=1}^p m(a_i)$.
\end{prs}

\begin{remark}
\emph{In fact Fan proved the above proposition for the projective fundamental group; however, the derivation for the affine group is trivial.}
\end{remark}

Eliyahu et al. \cite{ELST} proved the inverse direction to Fan's result (which was conjectured by Fan \cite{Fa2}), i.e. if the fundamental group of the arrangement is a direct sum of free groups and a free abelian group, then the associated graph has no cycles.

\medskip

We will generalize Fan's result to {\it real conic-line arrangements}. We start by defining them.

\begin{definition} \label{defCLarr}
A {\em real conic-line (CL) arrangement} $\A$ is a union of conics and lines in $\CC^2$, where all the conics and the lines are defined over $\RR$ and every singular point (with respect to a generic projection) of the arrangement is in $\RR^2$. In addition, for every conic $C \in \A$, $C \cap \RR^2$  is not an empty set, neither a point nor a (double) line.
\end{definition}

Moreover, we assume from now on the following assumption:
\begin{assumption}\label{assume}
Let  $\A$ be a real CL arrangement. Then, for each pair of components $h_1,h_2$ of $\A$, $h_1$ and $h_2$ intersect transversally (i.e. the intersection multiplicity of $h_1,h_2$ is 1 at each intersection point).
\end{assumption}
For example, a tangency point is not permitted.

\begin{remark} \label{remGenParab}
(1) We assume that no line passes through the branch
points of the conics with respect to a generic projection.

(2) As every singular point of a real CL arrangement is in $\RR^2$, the conics in the arrangements we deal with are either ellipses or hyperbolas, but not parabolas, since the second branch point of a parabola is at infinity.
\end{remark}
\medskip

Similar to Fan's graph for line arrangements, one can associate the following graph to a real CL arrangement:

 \begin{definition}\label{defGraph}\emph{
\emph{The graph} $G(\A)$ for a real CL arrangement $\A$ is defined as follows: its vertices  will be the multiple points (with multiplicity larger than 2), and its edges will be the segments {\it on the lines} connecting these points if two such points are  on the same line
(see an example in Figure \ref{graphEx}).}
\end{definition}

\begin{figure}[!ht]
\epsfysize 4cm
\centerline{\epsfbox{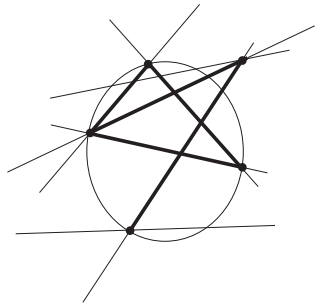}}
\caption{An example for the graph $G(\mathcal A)$ for a CL arrangement
$\mathcal A$.}\label{graphEx}
\end{figure}

\medskip

Then, one of the main results of this paper is:
\begin{thm}\label{main_result}
Let $\A$ be a real CL arrangement with one conic and $k$ lines, and $S=\{a_1, \dots, a_p,b_1,\dots,b_q\}$ be the set of all multiple points of $\A$, where the conic is passing through the intersection points $a_1, \dots, a_p$.  Suppose that $\beta (\A)=0$, where $\beta (\A)$ is the first Betti number of the graph $G(\mathcal A)$ (hence $\beta (\mathcal A)=0$
means that the graph $G(\mathcal A)$ has no cycles). Then:
$$\pi_1 (\CC ^2 - \mathcal A) \cong \ZZ^r \oplus \bigoplus\limits_{i=1}^p \FF_{m(a_i)-2} \oplus \bigoplus\limits_{i=1}^q \FF_{m(b_i)-1} ,$$
where $m(x)$ is the multiplicity of the singular point $x$ and\break $r=k+2p+q+1-\sum\limits_{i=1}^p m(a_i)- \sum\limits_{i=1}^q m(b_i)$.
\end{thm}

Note that while for line arrangements the inverse direction (i.e. such a structure of the fundamental group implies that the associated graph has no cycles) is correct \cite{ELST}, for CL arrangements it is not true anymore. For example, take three generic lines and a circle passing through the three intersection points (see Figure \ref{pi1Ab}(a)). Then, the fundamental group of the complement of this arrangement is abelian \cite{Deg}, although the first Betti number of the graph is $1$. We generalize this phenomenon in \cite{FG2}, showing that for a CL arrangement whose associated graph is a cycle of odd length and the conic passes through all the multiple points corresponding to the vertices of the cycle and all the multiple points have multiplicity 3, the corresponding fundamental group is abelian.

Note also that there are CL arrangements with a tangent point (which is excluded by our restrictions, see Assumption \ref{assume}) with an abelian fundamental group of the complement. For example, three lines and a conic tangent  only to one of the lines (see Figure \ref{pi1Ab}(b)) has an abelian fundamental group (see \cite{Deg}).

\begin{figure}[!ht]
\epsfysize 3.5cm
\epsfbox{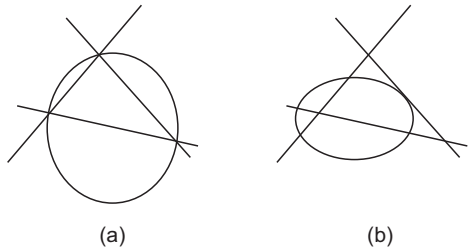}
\caption{CL arrangements with an abelian affine fundamental group $\Z^4$: arrangement (a) has $\beta (\A)=1>0$, and arrangement (b) has a tangency point.}\label{pi1Ab}
\end{figure}

\subsection{The braid monodromy of plane curves and the Zariski-van Kampen theorem}\label{bm_sec}

The reader who is familiar with the definition of the braid monodromy, its computation and the relevant Zariski-van Kampen theorem, can skip this subsection.

\medskip

We start by defining the braid monodromy associated to a plane curve.

\begin{definition}\label{defBraidGr}
\emph{Let $D$ be a closed disk in $ \mathbb{R}^2,$ and $K\subset \Int(D)$ a finite set of $n$ points.
The {\it braid group} $B_n[D,K]$ can
be defined as the group of equivalence classes of diffeomorphisms $\beta$
of $D$ such that $\beta(K) = K\,$ and $ \beta |_{\partial D} =
\text{Id}|_{\partial D}$, where two
diffeomorphisms are {\it equivalent} if they induce the same automorphism on $\pi_1(D - K,u)$.}
\end{definition}

Let $a,b\in K,$ and let $\sigma$ be a smooth simple path in
$\Int(D)$ connecting $a$ with $b$ \ such that $\sigma\cap K=\{a,b\}.$
Choose a small regular neighborhood $U$ of $\sigma$ contained in
$\Int(D),$ such that $U\cap K=\{a,b\}$. The
diffeomorphism of $D$ that switches the points $a$ and $b$ by a
counterclockwise $180^\circ$ rotation and is the identity on
$D - U$, defines an element of $B_n[D,K],$ called
{\it the half-twist defined by
$\sigma$} and denoted by $H(\sigma)$.

\medskip

\begin{definition} {The braid monodromy with respect to $C,\pi,u$.} \\
\emph{Let $C \subset \C^2$ be a curve. Choose a point $O \in \C^2, O \not\in C$, such that the projection $p: \C^2 \to\C^1 = \ell$  with a center $O$ (to a generic line $\ell$ called the \emph{reference line}) will be generic when restricting it to $C$ Denote $\pi = p|_C$ and
let $m=\deg\pi = \deg\, C$. Let $N=\{x\in \ell \bigm| \#\pi^{-1}(x)< m\}.$
Take $u\in \ell - N,$ 
and let  $\C^1_u=p^{-1}(u).$  There is a  naturally defined homomorphism:
$$\varphi: \pi_1(\ell-N,u) \rightarrow B_m[\C_u^1,\C_u^1\cap C],$$
which is called {\it the braid monodromy with respect to} $C,\pi,u$,  describing the motion of the points in the fiber
(see \cite{MoTe1})}.\end{definition}

In fact, letting $E$ be a big disk in $\ell = \C^1$ such that $N \subset E$, we can
also choose the path in $E- N$ not to be a loop, but just a
non-selfintersecting path. This induces a diffeomorphism between
the models $(D,K)$ at the two ends of the considered path, where
$D$ is a big disk in $\C^1_u$, and $K = \C_u^1\cap C \subset
D$.

\begin{definition} \label{defLefdif}  ${\psi_T, \text{ the Lefschetz diffeomorphism induced by a path} \ T }$.\\
\emph{Let $x_0,x_1 \in E - N$ be two different points,
 $T:[0, 1]\ri E- N$ a non-selfintersecting path in $E - N$ connecting $x_0$ with
$x_1$. There exists a continuous
family of diffeomorphisms $\psi_{(t)}: D\ri D,\ t\in[0,1],$ such
that $\psi_{(0)}={\rm Id}$, $\psi_{(t)}(K(x_0))=K(T(t)) $ for all
$t\in[0,1]$, and  $\psi_{(t)}(y)= y$ for all $y\in \partial D$.
For emphasis, we write $\psi_{(t)}:(D,K(x_0))\ri(D,K(T(t)))$. The }
Lefschetz diffeomorphism induced by a path $T$ \emph{is the
diffeomorphism:
$$\psi_T= \psi_{(1)}: (D,K(x_0))\usr (D,K(x_1)).$$
Since $ \psi_{(t)} \left( K(x_{0}) \right) = K(T(t))$ for all $t\in
[0,1]$, we have a family of canonical isomorphisms:
$$\psi_{(t)}^{\nu}: B_m\left[ D, K(x_{0})\right] \usr B_m\left[
D, K({T(t)})\right], \ \quad \text{for all} \, \,
t\in [0,1].$$}
\end{definition}

Let  $\{\Gamma_i\}$ be a geometric (free) base (called a {\it g-base}) of $\p(\C^1 - N, u)$ (see \cite{MoTe1} for the exact definition), $\vp$ the braid monodromy of $C , \vp:\p(\C^1 - N, u) \rightarrow B_m$. In order to find out a presentation of the fundamental group of the complement of $C$ in $\CC^2$, we have to find out what are $\vp
(\Gamma_i),$ for all $i$. We refer the reader to the definition of
a \textit{skeleton} $\lambda_{x_j}$, for all $x_j \in N$ (see \cite{MoTe2}),
which is a model of a set of consecutive paths connecting points in the fiber,
which coincide when approaching
$A_j=$($x_j,y_j$)$\in C$ from the
right. To describe this situation in more details, for $x_j \in
N$, let $x_j' = x_j + \alpha$, where $0 < \alpha \ll 1$. The skeleton of $x_j$ is defined
as a system of consecutive segments connecting the points in $K(x_j') \cap
D(A_j,\varepsilon)$, where $0 < \alpha \ll \varepsilon \ll 1$ and
$D(A_j,\varepsilon)$ is a disk centered in $A_j$ of radius $\varepsilon$.

For a given skeleton, denote by
$\Delta\langle\lambda_{x_j}\rangle$ the braid which rotates a small neighborhood of the given
skeleton by 180$^\circ$ counterclockwise. Note that if $\lambda_{x_j}$ is a single path, then
$\Delta\langle\lambda_{x_j}\rangle = H(\lambda_{x_j})$.

We also refer the reader to the definition of $\delta_{x_0}$, for
$x_0 \in N$ (see \cite{MoTe2}), which describes the Lefschetz
diffeomorphism induced by a path going below $x_0$, for different
types of singular points (either a transversal intersection of several lines at a point or a branch point; for example, when
going below a node the corresponding Lefschetz diffeomorphism is a half-twist of the corresponding skeleton).

Thus, the Lefschetz diffeomorphism induced by a path going from $x_j'$ to $u$ below the points $x_i$, $1 \leq i \leq j-1$, is the composition
of the corresponding $\delta_{x_i}$'s, i.e. $\prod\limits_{m=j-1}^{1}\delta_{x_m}$ \cite{MoTe1,MoTe2}. We  illustrate the action of a specific Lefschetz diffeomorphism (induced by a line arrangement) in the following example.

\medskip

\begin{example}
We present here an example for computing a skeleton and the effect of applying a Lefschetz diffeomorphism on it (more examples can be found in \cite{EGT1,MoTe1,MoTe2}).

\begin{figure}[!ht]
\epsfysize 6.5cm
\epsfbox{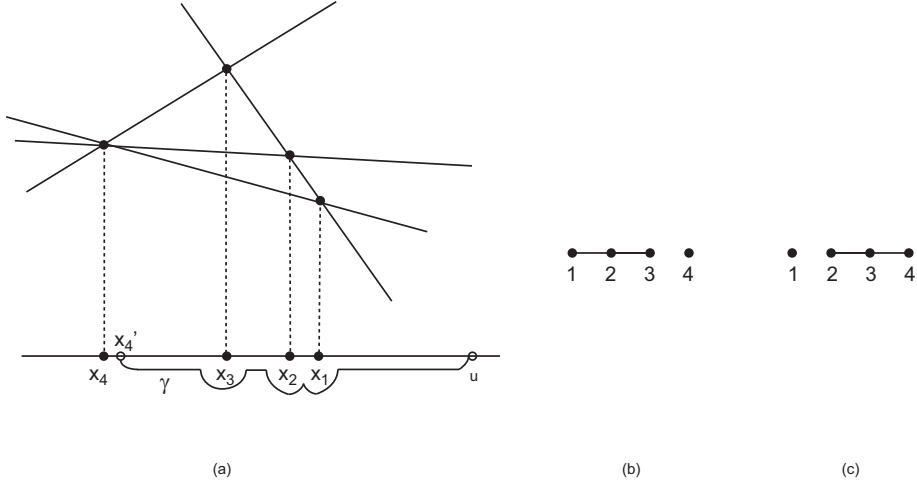}
\caption{Part (a) is an example of a line arrangement, Part (b) is the initial skeleton of the point $x_4$ (i.e. in the fiber over $x_4'$) and Part (c) is its final skeleton (in the fiber over $u$).}\label{exampleSkel}
\end{figure}

$\lambda_{x_4}$, the initial skeleton of the point $x_4$ in Figure \ref{exampleSkel}(a) (i.e. in the fiber over $x_4'$), is presented in Figure \ref{exampleSkel}(b).

$\delta$, which is the Lefschetz diffeomorphism induced by $\gamma$ (a path going from $x_4'$ to $u$), is:
$$\Delta \langle 3,4\rangle \Delta \langle2,3 \rangle \Delta \langle 1,2 \rangle.$$

$\lambda_{x_4}$, the final skeleton of the point $x_4$ in Figure \ref{exampleSkel}(a) (i.e. in the fiber over $u$ after applying $\delta$ on $\lambda_{x_4}$), is presented in Figure \ref{exampleSkel}(c).
\end{example}

\medskip

Based on the braid monodromy, we can compute presentations for the groups $\pi_1(\CP^2 -\overline C)$ and $\pi_1(\C^2 - C)$ (where  $ \overline C \subset \CC \PP^2$ is a projective curve, $C = \overline C \cap \C^2$).

Let $\{\G_i\}$ be a $g$-base of $G = \pi_1(\C^1_u-(\C^1_u \cap C),u),$ where $\C^1_u = \pi^{-1}(u)
= \C \times \{ u \}$. Then, $\pi_1(\C^2-C,u)$ is generated by the images of $\{\G_i\}$ in
$\pi_1(\C^2-C,u)$. We use now the
Zariski-van Kampen theorem \cite{vK} in order to
compute the relations between the generators of $G.$
The theorem essentially says that every singular point (with respect to a projection from $O$ to $\ell$) induces a relation in $\pi_1(\C^2 - C)$, and these induced relations are all the relations of $\pi_1(\C^2 - C)$.

Since we are dealing only with CL arrangements, we formulate the theorem only for a curve having only branch points, nodes and multiple intersection points as singular points (with respect to a projection).

\medskip

\begin{thm}[Zariski-van Kampen \cite{vK}] \label{vk_thm} Let
$\overline C$ be a CL arrangement in $\CP^2$ and
$C=\C^2\cap\overline C$. Let $s$ be the number of singular points of $C$ with respect to the projection from $O$.
 For every $j \in \{ 1,\dots ,s \}$, consider the skeleton:
 $$ \lambda'_{x_j} = \left\langle (\lambda_{x_j})\bigg(\prod\limits_{m=j-1}^{1}\delta_{x_m}\bigg)\right\rangle.$$
Then, $\pi_1(\C^2-C,u)$ is generated by the images of $\{\G_i\}$ in
$\pi_1(\C^2-C,u)$ and the only relations are those induced by the skeletons
$\lambda'_{x_j}$ in the following way:

If the point $x_j$ is either a node or a branch point, then the skeleton $\lambda'_{x_j}$ is a path connecting two points.
In this case, the relation is either $a_1a_2=a_2a_1$ (for a node) or $a_1=a_2$ (for a branch point), where the computation of the  $a_i$'s will be described after the theorem.

If the point $x_j$ is a intersection point of multiplicity $k$, then the skeleton $\lambda'_{x_j}$ is a set of $k-1$ consecutive paths connecting $k$ points.
In this case, the relations are:
$$a _k a_{k-1} \cdots  a _1 = a _1 a _k a_{k-1} \cdots   a _2 = \cdots =a _{k-1} a _{k-2} \cdots a _1 a _k,$$
where the computation of the $a_i$'s will be described after the theorem.
\end{thm}

\begin{notation}
The relations
$$a _k a_{k-1} \cdots  a _1 = a _1 a _k a_{k-1} \cdots   a _2 = \cdots =a _{k-1} a _{k-2} \cdots a _1 a _k,$$
will be denoted in an abbreviated form as:
$$
[a_1,a_2,\ldots,a_k]=e.
$$
\end{notation}

We start by describing the $a_i$'s in the case that the skeleton is a path
connecting two points, i.e. the singular point is either a node or a branch point.
Let $D$ be a disk circumscribing the skeleton, and let $K$ be the set of points.
 Choose an arbitrary point on the path and `pull' it down to $\partial D$,
splitting the path into two parts, which are connected at one end
to $u_0 \in \partial D$ and at the other end to the two endpoints of the path in $K$.
The loops associated to these two paths are elements in the group
$\pi_1 (D-K,u_0)$ and we call them $a_1$ and $a_2$. The corresponding
elements commute (in the case of a node) or equal (in the case of a branch point) in the fundamental group of the arrangement's complement. Figure \ref{av_bv} illustrates this procedure.

\begin{figure}[!ht]
\epsfysize 6cm
\epsfbox{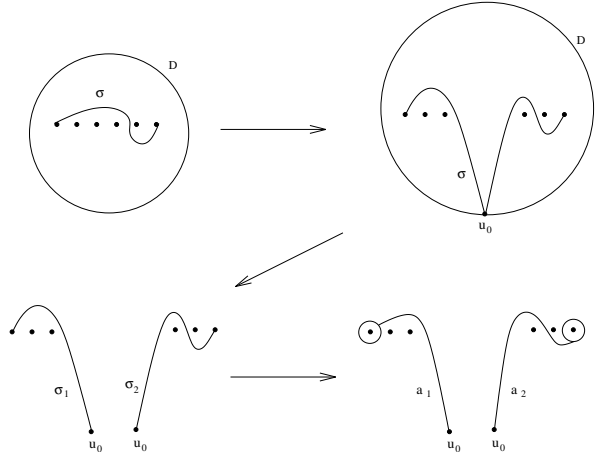}
\caption{Construction of $a_1,a_2$ for a node or a branch point.}\label{av_bv}
\end{figure}

Now we show how to write $a_1$ and $a_2$ as words in the
generators
$\{\Ga_1, \dots, \Ga_\ell\}$ of $\pi _1(D-K,u_0)$. We start
with the generator corresponding to the endpoint of
$a_1$ (or $a_2$), and conjugate it as we move along $a_1$ (or
$a_2$) from its endpoint in $K$ to $u_0$ as follows: for every
point $i \in K$ which we pass from above, we conjugate by $\Ga_i$
while moving from left to right, and by $\Ga_i^{-1}$ while moving
from right to left.

For example, in Figure \ref{av_bv},
$$a_1 = \Ga_3 \Ga_2 \Ga_1 \Ga_2^{-1} \Ga_3^{-1}, \quad a_2 = \Ga_4 ^{-1} \Ga_6 \Ga_4.$$
Assuming that the singular point is a node, the induced relation is the following commutative relation:
$$\Ga_3 \Ga_2 \Ga_1 \Ga_2^{-1} \Ga_3^{-1} \cdot \Ga_4 ^{-1} \Ga_6 \Ga_4 =\Ga_4 ^{-1} \Ga_6 \Ga_4 \cdot \Ga_3 \Ga_2 \Ga_1 \Ga_2^{-1} \Ga_3^{-1}.$$

One can check that the induced relation is independent of the point in
which the path is split.

\bigskip

For an intersection point of multiplicity $k$, we compute the elements in the
group
$\pi_1 (D-K,u_0)$ in a similar way, but the induced relations are
 of the following cyclic type:
$$a _k a_{k-1} \cdots  a _1 =
 a _1 a _k  a_{k-1}\cdots   a _2 = \cdots =
 a _{k-1} a _{k-2} \cdots a _1 a _k.$$
For computing the $a_i$'s, we choose an
arbitrary point on one of the paths and pull it down to
$u_0$.
For each of the $k$ points of the skeleton, we generate the
loop associated to the path from
$u_0$ to that point, and translate this path to a word in
$\G_1,\dots,\G_\ell$ by the procedure described above (note that while computing the word corresponding to $a_j$, the path $a_j$ is considered as being below the final points of $a_i$, $i \neq j$).

\begin{figure}[!ht]
\epsfysize 6cm
\epsfbox{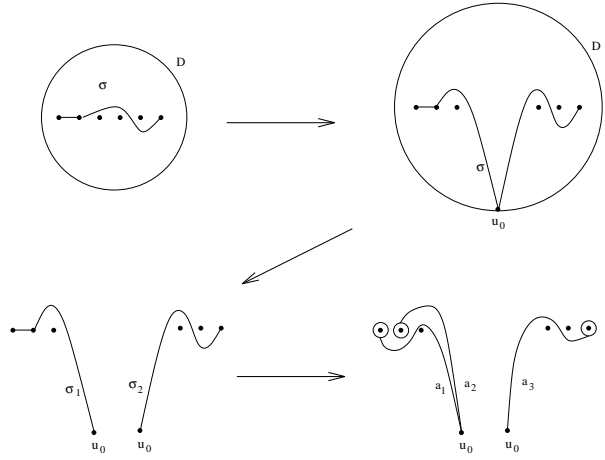}
\caption{Construction of $a_1,a_2,a_3$ for a multiple intersection point}\label{av_bv_mul}
\end{figure}

In the example given in Figure \ref{av_bv_mul}, we have:
$$a_1 =
\G_3 \G_1 \G_3^{-1}, a_2 = \G_3 \G_2
\G_3^{-1} \mbox{ and } a_3 = \G_4^{-1} \G_6 \G_4,$$ so the induced relations are:
\begin{eqnarray*}
\G_4^{-1} \G_6 \G_4 \cdot \G_3 \G_2 \G_3^{-1} \cdot \G_3 \G_1
\G_3^{-1} & = & \G_3 \G_1 \G_3^{-1} \cdot \G_4^{-1} \G_6 \G_4 \cdot
\G_3 \G_2 \G_3^{-1}\\
& = & \G_3 \G_2 \G_3^{-1} \cdot \G_3 \G_1
\G_3^{-1} \cdot \G_4^{-1} \G_6 \G_4.
\end{eqnarray*}

\subsection{The conjugation-free property} \label{sec_conj_free_prop}

 In this section, we define the notion of a {\it conjugation-free geometric presentation} for the fundamental group of line and CL arrangements, following the  definition given in \cite{EGT1}:
\begin{definition}\label{CFGP-CLArr1}
Let $G$ be a fundamental group of the affine or projective complements of a real CL  arrangement with $k$ lines and $n$ conics (where $k>0$ and $n \geq 0$). We say that $G$ has {\em a conjugation-free geometric presentation} if $G$ has a presentation with the following properties:
\begin{itemize}
\item In the affine case, the generators $\{ x_1,\dots, x_{k+2n} \}$ are the meridians of lines and conics at $\C^1_u = \pi^{-1}(u)$, and therefore there are $k+2n$ generators.
\item In the projective case, the generators are the meridians of lines and conics at $\C^1_u = \pi^{-1}(u)$ except for one, and therefore there are  $k+2n - 1$ generators.
\item In both cases, the induced relations are of the following types:
$$x_{i_t} x_{i_{t-1}} \cdots x_{i_1} = x_{i_{t-1}} \cdots x_{i_1} x_{i_t} = \cdots = x_{i_1} x_{i_t} \cdots x_{i_2}$$
induced by an intersection point of multiplicity $t$, or
$$x_{i_1}=x_{i_2},$$
induced by a branch point,
where $\{ i_1,i_2, \dots , i_t \} \subseteq \{1, \dots, m \}$ is an increasing subsequence of indices,
where $m=k+2n$ in the affine case and $m=k+2n-1$ in the projective case. Note that if $t=2$ in the first type, we get the usual commutator.
\item In the projective case, we have an extra relation that a specific multiplication of all the generators is equal to the identity element.
\end{itemize}
Note that in each case we claim that with respect to  particular choices of the reference line $\ell$ (i.e. the line to which we project the arrangement), the point $u$ (the basepoint for both  the meridians in the fiber $\CC^1_u$ and the loops in the group $\pi_1(\ell-N,u)$)
and the projection point $O$, we have this conjugation-free property.
\end{definition}

\begin{remark}  \label{remProveLater}\emph{
In the model we work with, the reference line $\ell$ is $\ell = \{y = a\}, a\ll 0$, where $\ell$ is chosen to be below all the real singular points of the arrangement, the projection in $\CC^2$ is $(x,y) \rightarrow x$ (i.e. in $\CP^2$, the point $O$ is $(1:0:0)$) and the point $u \in \ell$ is always a real point, i.e. $u = (u_x,a)$, where $u_x \in \RR$.}
\end{remark}

As will be proven below, for certain families of line arrangements the property of having a  conjugation-free geometric presentation is independent
of the choice of $u$. We thus conjecture the following:

\begin{conjecture}
The property of having a conjugation-free geometric presentation is independent of the choices of $O$, $\ell$ and $u \in \ell$.
\end{conjecture}

\begin{remark}\emph{
The notion of a conjugation-free geometric presentation for the fundamental group can be generalized to any arrangement of plane curves (with the proper modifications with respect to the degrees of the curves and the types of singularities).}
\end{remark}

Note that the importance of the family of CL arrangements whose fundamental group has a conjugation-free geometric presentation is that such a presentation of the fundamental group can be read directly from the arrangement without any additional computation, and thus depends on the combinatorics of the arrangement.

\medskip

The goal of the next two sections is to give a complete proof to the following proposition:

\begin{prs} \label{lemAddLineComProof}
(1) Let $\LL$ be a real line arrangement such that  $\pi_1(\CC^2 - \LL,u)$ has a conjugation-free geometric presentation \emph{for any real basepoint} $u \in \ell - N$ (where $N$ is the set of the projection of singular points with respect to the projection $\pi$). Let $L$ be a real line not in $\LL$ that passes through a single intersection point of $\LL$. Then  $\pi_1(\CC^2 - (\LL \cup L),u)$ has a conjugation-free geometric presentation for any real basepoint $u$.

(2) Let $\A$ be a real CL arrangement with one conic such that  $\pi_1(\CC^2 - \A,u)$ has a conjugation-free geometric presentation \emph{for any real basepoint} $u \in \ell - N$. Let $L$ be a real line not in $\A$ that passes through a single intersection point of $\A$ such that $\beta(\A \cup L) = 0$. Then  $\pi_1(\CC^2 - (\A \cup L),u)$ has a conjugation-free geometric presentation for any real basepoint $u$.
\end{prs}


\medskip

Although a proof to Proposition \ref{lemAddLineComProof}(1) was already given in \cite[Proposition 2.2]{EGT2}, it is not a complete proof: the authors of
\cite{EGT2} assume implicitly that certain operations (such as a rotation of the arrangement, moving a singularity through ``infinity" or a rotation of a line in the arrangement), that are performed on a line arrangement $\LL$, preserve the conjugation-free
property. Indeed, while these operations induce \emph{isomorphic} fundamental groups (as partially proved in \cite{GTV}), one still has to prove that if $\beta$
is such an operation and if $\pi_1(\CC^2 - \LL)$ has a conjugation-free geometric presentation with respect to the indicated model above
(see Definition \ref{CFGP-CLArr1} and Remark \ref{remProveLater}), then $\pi_1(\CC^2 - \beta(\LL))$ has also a conjugation-free geometric presentation with respect to the \emph{same model}.

\section{Preservation of Conjugation-Freeness: The case of real line arrangements}\label{secRealLineArr}

The goal of this section is to give a complete proof to Proposition \ref{lemAddLineComProof}(1) (which corrects the proof given in \cite{EGT2}): given a  real line arrangement $\LL$ whose fundamental group has a conjugation-free presentation, then, under certain conditions, passing a real line $L$ through at most one singular point of $\LL$, the new arrangement $\LL \cup L$
has the same property of conjugation-freeness. We divide the proof  into a sequence of lemmata, as
we have to prove that the following operations preserve the conjugation-free property:


\begin{itemize}
\item $h_1 : $ given a line $\ell \in \LL$ that passes through a single intersection point $p \in \LL$, rotate the line around $p$ as long as it does not coincide with a different line.

\item $h_2 : $ given a line $\ell \in \LL$ that passes through a single intersection point $p \in \LL$, move  the line over a different line $\ell'$ that passes also through $p$ and through another multiple point (i.e. the deformation takes place in $\CC^2$, see  Figure \ref{triangle-line}).

\begin{figure}[!ht]
\epsfysize 3cm
\epsfbox{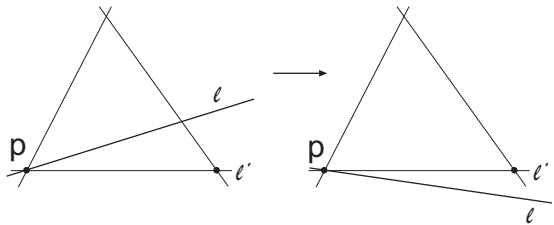}
\caption{Moving a line $\ell$ over another line $\ell'$.}\label{triangle-line}
\end{figure}

 \item $h_3 : $ Changing the basepoint $u \in \ell - N$ when computing $\pi_1(\CC^2 - \LL,u)$.
\end{itemize}

Having all these operations preserving the conjugation-free property, we can indeed claim (as was claimed in the proof
of \cite[Proposition 2.2]{EGT2}) that after adding the new  line, we still have that the fundamental  group of the new arrangement has a conjugation-free geometric presentation.

\begin{remark}
\rm{(1) Comparing to \cite{EGT2}, we reduced the number of operations (performed on $\LL$) preserving the conjugation-free property. We are not
proving that a rotation (of the whole arrangement) or the moving of a singular point through ``infinity" preserve
the conjugation-free property; instead of these operations, we only prove that changing the real basepoint $u \in \ell - N$
preserves this  property. However,  Lemma \ref{rot_op_CFpreserve} below deals with  the preservation of the conjugation-free property
during rotating certain classes of line arrangements.

(2) Note also that it is nor clear if the operation of moving a singular point through ``infinity", that was presented in \cite{GTV} in the context of wiring diagrams, can be applied for line arrangements. Indeed, if one of the lines that goes through an intersection point $p$, goes over also through many other multiple intersection points, it is not clear that one can move $p$ through ``infinity" without changing the intersection lattice with respect to this line.}
\end{remark}

Thus, given a real line arrangement $\LL$, we have to answer the following two questions: first, what is the explicit isomorphism
$f_i: \pi_1(\CC^2 - \LL) \rightarrow \pi_1(\CC^2 - h_i(\LL))$ for $1 \leq i \leq 3$? and second, does this isomorphism preserve
the conjugation-free property?

\begin{remark}
Note that from now on, whenever we refer to the basepoint $u \in \ell$, it is implicitly assumed that $u$ is real.
\end{remark}

\begin{lemma} \label{g1_op_CFpreserve}
The operation $h_1$ preserves the conjugation-free property.
\end{lemma}

\begin{proof}
Let $\LL$ be a real line arrangement. The operation $h_1$ - rotating a line over a multiple point - and its effects on the fundamental group is described in \cite[Section 4.5]{GTV}. In \cite[Theorem 4.13]{GTV} it is proven that the isomorphism $f_1: \pi_1(\CC^2 - \LL,u) \rightarrow \pi_1(\CC^2 - h_1(\LL),u)$ is given by $\G_j \mapsto \G'_j$, where $\G_j$ (resp. $\G_j'$) are the geometric generators for $\pi_1(\CC^2 - \LL)$ (resp. $\pi_1(\CC^2 - h_1(\LL))$) and $u \in \ell - N$ is a basepoint located to the right of all the points in $N \subset \ell$. Since $f_1$ sends generators to generators, it is obvious that the relations in $\pi_1(\CC^2 - h_1(\LL))$ have no conjugations, and hence the new presentation is still conjugation-free.
\end{proof}

\begin{lemma} \label{g2_op_CFpreserve}
The operation $h_2$ preserves the conjugation-free property.
\end{lemma}

\begin{proof}
Let $\LL$ be a real arrangement of $n$ lines, $\ell_b \in \LL$ a line that passes through a single multiple point $p \in \LL$ and $\ell_a \in \LL$ a different line  that passes through  $p$ and through another multiple point $p'$, see Figure \ref{RotateLineOverLine}(a).
 Note that $a,b \in \{1,\ldots,n\}$ are the numerations of the lines the fiber over the basepoint $u$.

 First, by Lemma \ref{g1_op_CFpreserve}, we can rotate the line
$\ell_b$ (where the center of rotation is $p$) till it is very close to $p'$, not coinciding with $\ell_a$, see Figure \ref{RotateLineOverLine}(b). Assuming that $x(p) \ll u$,  we note that while numerating the points in the fiber over $u$, $b = a+1$.
Then, we rotate the line $\ell_b$ over $\ell_a$, such that in the fiber $\CC^1_u$, the corresponding point $p_b$ will perform a $180^{\circ}$
clockwise rotation and the point $p_a$ will remain unchanged (see Figure \ref{RotateLineOverLine}(c)).

\begin{figure}[!h]
\epsfysize 6cm
\epsfbox{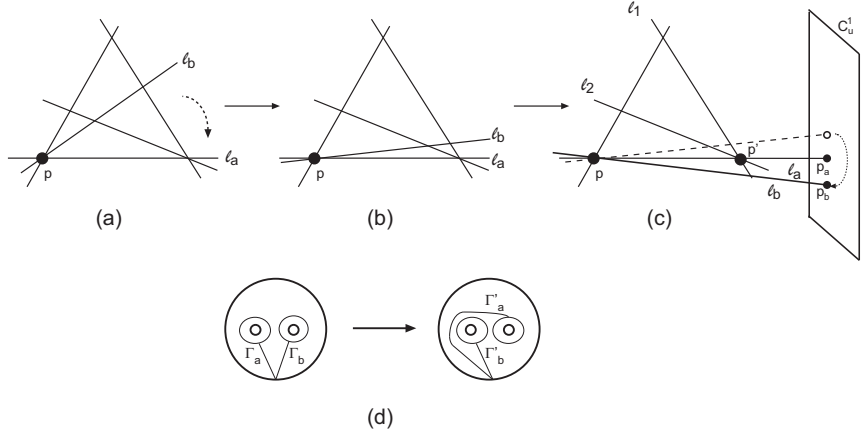}
\caption{The rotation of the line $\ell_b$ over another line $\ell_a$ (Parts (a)-(c)) and its effect on the geometric generators (Part (d)).}\label{RotateLineOverLine}
\end{figure}

The induced isomorphism $f_2: \pi_1(\CC^2 - \LL,u) \rightarrow \pi_1(\CC^2 - \h_2(\LL),u)$ is given by
the following map:
$$
\G_b \mapsto \G'_a,\,$$\begin{equation} \label{eqnRelAfterIso} \G_a \mapsto {\G'_a}^{-1} \G'_b \G'_a, \, \end{equation}$$ \G_i \mapsto \G'_i, \,\, i \neq a,b,$$
(see Figure \ref{RotateLineOverLine}(d)).
Note that the only relations in $\pi_1(\CC^2 - \h_2(\LL),u)$ that presumably have conjugations  are those
that  involved the generator  $\G_a$ in $\pi_1(\CC^2 - \LL,u)$, i.e.  the relations induced by the intersection points that are on $\ell_a$.

The relation induced by the point $p$ in $\pi_1(\CC^2 - \LL,u)$ is $$[\G_{i_1},\ldots,\G_{i_m},\G_a,\G_b,\G_{i_{n+3}},\ldots,\G_{i_k}]=e$$ (where $k$ is the multiplicity of $p$), which is
mapped  in $\pi_1(\CC^2 - \h_2(\LL),u)$ to:

\begin{equation}\label{eqnRelAfterIsoMoveLine}
[\G'_{i_1},\ldots,\G'_{i_m},{\G'_a}^{-1} \G'_b \G'_a,\G'_a,\G'_{i_{n+3}},\ldots,\G'_{i_k}]=e.
\end{equation}

Now, after the rotation of the line $\ell_b$, the numeration of the lines passing through $p$ is as follows:
  $$
  i_1,\ldots,i_m,a,b=a+1,i_{n+3},\ldots,i_k,
  $$ which means that we need to prove that  the following relation holds in $\pi_1(\CC^2 - \h_2(\LL),u)$:
$$[\G'_{i_1},\ldots,\G'_{i_m},\G'_a,\G'_b,\G'_{i_{n+3}},\ldots,\G'_{i_k}]=e.$$

 However, an easy check (i.e. expanding the brackets of relation (\ref{eqnRelAfterIsoMoveLine})) shows that this is indeed the  relation we get from
 relation (\ref{eqnRelAfterIso}). For example, for $k=3$, we have that
$$
[\G_{i_1},\G_a,\G_b]=e \,\,\rightarrow\,\, [\G'_{i_1},{\G'_a}^{-1} \G'_b \G'_a,\G'_a]=e.
$$
Therefore, the right hand side is equivalent to the following relations:
$$
{\G'_a}^{-1} \G'_b \G'_a \cdot \G'_{i_1} \cdot \G'_a  = \G'_{i_1} \cdot \G'_a  \cdot {\G'_a}^{-1} \G'_b \G'_a  \,\, \Rightarrow
$$
\begin{equation} \label{eqnRelh2}
 \G'_b \G'_a \G'_{i_1} = \G'_a \G'_{i_1} \G'_b
\end{equation}
and
$$
\G'_a \cdot {\G'_a}^{-1} \G'_b \G'_a \cdot \G'_{i_1} = {\G'_a}^{-1} \G'_b \G'_a  \cdot \G'_{i_1}  \cdot \G'_a \,\,\Rightarrow  \,\, \G'_a \G'_b \G'_a \G'_{i_1} = \G'_b \G'_a \G'_{i_1} \G'_a
$$
\begin{equation} \label{eqnRelh2_2}
\overset{\rm{Eqn.}\,\,\, (\ref{eqnRelh2})}{\Rightarrow}\,\, \G'_a \G'_b \G'_a \G'_{i_1} =  \G'_a \G'_{i_1} \G'_b \G'_a \,\,\Rightarrow \,\, \G'_b \G'_a \G'_{i_1} = \G'_{i_1} \G'_b \G'_a.
\end{equation}
Thus, relations (\ref{eqnRelh2}) and (\ref{eqnRelh2_2}) imply that $[\G'_{i_1},\G'_a,\G'_b]=e$.

\medskip

Let $p'$ be another intersection point on $\ell_a$ (see Figure \ref{RotateLineOverLine}(c)). Assume for simplicity that only two more lines $\ell_1,\ell_2$ pass through $p'$ and that the  relation induced by $p'$ in $\pi_1(\CC^2 - \LL,u)$ is $[\G_1,\G_a,\G_2]=e$. As $\ell_b$ passes through only one multiple point $p$, it intersects
$\ell_1,\ell_2$ transversally, and therefore the induced relations in $\pi_1(\CC^2 - \LL,u)$ (which do not have conjugations since the presentation is conjugation-free) are $[\G_1,\G_b]= [\G_2,\G_b]=e$.
Thus, in $\pi_1(\CC^2 - \h_2(\LL),u)$, the corresponding induced relations are:
$$
[\G'_1,\G'_a]= [\G'_2,\G'_a]=e
$$
and hence:
$$
[\G'_1,{\G'_a}^{-1} \G'_b \G'_a,\G'_2]=e \overset{[\G'_1,\G'_a]= [\G'_2,\G'_a]=e}{\Rightarrow} [\G'_1, \G'_b,\G'_2]=e.
$$

Therefore, the  relation induced by $p'$ has no conjugations; hence the new presentation is also conjugation-free, as needed.

Note that the last part of the proof is independent of the order of the generators (i.e. how the lines $\ell_1,\ell_2$ pass through $p'$)
and is similar when either only one line passes through the point $p'$ or more than three lines  pass through this point.
\end{proof}

We now treat the operation $h_3$. We start with a remark.

\begin{remark} \label{remG3_Action}
\rm{

(I) Let $p$ an intersection point of $\LL$.  Changing the basepoint from  $u = x(p)+\varepsilon$, where $0<\varepsilon \ll 1$, to $u' = x(p)-\varepsilon$ (see Figure \ref{changingBasePoint})
 induces an isomorphism on the geometric generators, which is described below.

\begin{figure}[!ht]
\epsfysize 4cm
\epsfbox{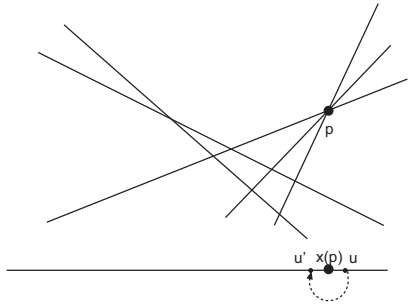}
\caption{Changing the location of the basepoint from the right of an intersection point to its left.}\label{changingBasePoint}
\end{figure}

  If the local skeleton at the point $p$ is composed of
 the lines numerated by $a,a+1,\ldots,a+s$, and the geometric generators of the group  $\pi_1(\CC^2 - \LL,u)$ (resp. $\pi_1(\CC^2 - \LL,u')$) are $\G_i$ (resp. $\G'_i$), then
 the isomorphism $f_3:\pi_1(\CC^2 - \LL,u) \rightarrow \pi_1(\CC^2 - \LL,u')$ is given by:

 \begin{center}
\begin{equation}\label{eqnIsoF3}f_3(\Ga _i) =
\left\{ \begin{array}{ccl}
{\Ga _i'\, }          &&  { 1 \leq i < a } \\
{{\Ga' _a} ^{-1} \cdots {\Ga'^{-1}_{a+s-1}} \Ga' _{a+s} \Ga'_{a+s-1} \cdots  \Ga_a'\,}  && {i=a} \\
{{\Ga' _a} ^{-1} \cdots {\Ga'^{-1}_{a+s-2}} \Ga_{a+s-1}' \Ga_{a+s-2}' \cdots \Ga _a'\,} && {i=a+1} \\
{\vdots} && {\,\,\,\,\,\,\vdots} \\
{{\Ga' _a} ^{-1} \Ga_{a+1}' \Ga _a'\,} && {i=a+s-1} \\
{\Ga _a'\, } && {i=a+s} \\
{\Ga _i' \,}          &&  {  a+s < i \leq \ell }
\end{array}
\right.
\end{equation}

\end{center}

(II) By the the isomorphism presented in (\ref{eqnIsoF3}), one can show that if $p$ is a node, then $f_3$ indeed preserves the conjugation-free property. Indeed, if the basepoint is  $u'$, the relation $[\G'_a,\G'_{a+1}]=e$ exists. In this case, the isomorphism $f_3$  is
$$
\G_a \mapsto \G'_{a+1},$$$$ \G_{a+1} \mapsto {\G'^{-1}_{a+1}}\G'_{a}\G'_{a+1} \overset{[\G'_a,\G'_{a+1}]=e}{=} \G'_{a}
$$
and since any geometric generator is sent to a geometric generator, $f_3$ preserves the conjugation-free property.}

\end{remark}

\begin{lemma} \label{g3_op_CFpreserve}
Let $\LL$ be a real line arrangement which can be presented as a union  $\LL = \LL' \cup L$, where $\LL'$ is a real line arrangement whose fundamental group has a conjugation-free geometric presentation
for \emph{every real} basepoint $u \in \ell-N$ and $L \not\in \LL'$ is a real line that passes through at most  one intersection point $p$ of $\LL'$.
 Then, on this class of real line arrangements,
the operation $\h_3$ preserves the conjugation-free property for $\LL$; that is, the fundamental group $\pi_1(\CC^2 - \LL,u)$ has the conjugation-free property for every real basepoint $u \in \ell-N$.
\end{lemma}

\begin{remark} \label{rem_NotRest}
\emph{Though the condition on $\LL$ in Lemma \ref{g3_op_CFpreserve} seems restrictive, one should note that if an arrangement  $\LL$ can be built
by adding  one line that passes through at most  one intersection point at each step of the construction, then $\LL$ has eventually a conjugation-free geometric presentation, independent of the chosen real basepoint $u$ (such an arrangement has a conjugation-free graph, see Section \ref{secCFGraphs} for more details). Indeed, note that for the arrangement $\LL_0$
that has only nodes, $\pi_1(\CC^2 - \LL_0,u)$ is abelian and therefore every relation has no conjugations, which means that it is
independent of the basepoint $u$. Then, one can build $\LL$ from $\LL_0$ by adding only one line at  each step (for an exact definition of $\LL_0$ and an example of this process see the proof of Proposition \ref{prsOnlyCycle} and Section \ref{subsecCFGline}). Thus, by Lemma \ref{g3_op_CFpreserve}, proving that $\pi_1(\CC^2 - \LL,u)$ has a conjugation-free geometric presentation, independent of the location of the real basepoint $u$ on $\ell$.}
\end{remark}

\begin{proof} We remind again that whenever we refer to the basepoint $u$, we implicitly assume that $u \in \CC^2$ is real (in its two coordinates). Moreover, since we choose the reference  line $\ell$ as $\ell = \{y = a : a \in \RR \}, a \ll 0$, the $y$-coordinate of $u$ is fixed. Therefore, when we write, for example, $u = x(a) + \varepsilon$, we mean that $x(u) = x(a) + \varepsilon$. For the convenience of the reader, we omit the notation $x(u)$ and write only $u$.

First, note that our assumption implies that for every basepoint $u_0 \in \ell -N$, $\pi_1(\CC^2 - \LL',u_0)$  has a conjugation-free geometric presentation.

If  $L$ intersects every line of $\LL'$ transversally, then, if
$$\pi_1(\CC^2 - \LL',u_0) = \langle x_1,\ldots,x_k : R  \rangle$$ is a conjugation-free presentation (where $R$ is a complete set of relations, which have no conjugations), then, denoting by $z$ the generator corresponds to $L$ we have  by \cite{OkSa} that: $$\pi_1(\CC^2 - \LL,u_0) = \ZZ \oplus \pi_1(\CC^2 - \LL',u_0) = \langle z,x_1,\ldots,x_k : R , [z,x_i]=e, \, 1 \leq i \leq k  \rangle,$$ which is obviously also a conjugation-free presentation too.

\medskip

 Assume thus that $\LL$ does not intersect every line of $L$ transversally. Let $p$ be an intersection point of $\LL'$. We draw $L$ through $p$ and assume that $L$ is almost vertical to the reference line $\ell$ with a very negative slope. We can assume that, since after we prove that the arrangement $\LL' \cup L$ has a conjugation-free geometric presentation, we can use Lemma \ref{g1_op_CFpreserve} to rotate the line $L$ around the point $p$ (back to its ``original" place). The arrangement is now at the position described in  Figure \ref{addingLine3}(a).

\begin{figure}[!ht]
\epsfysize 6.5cm
\epsfbox{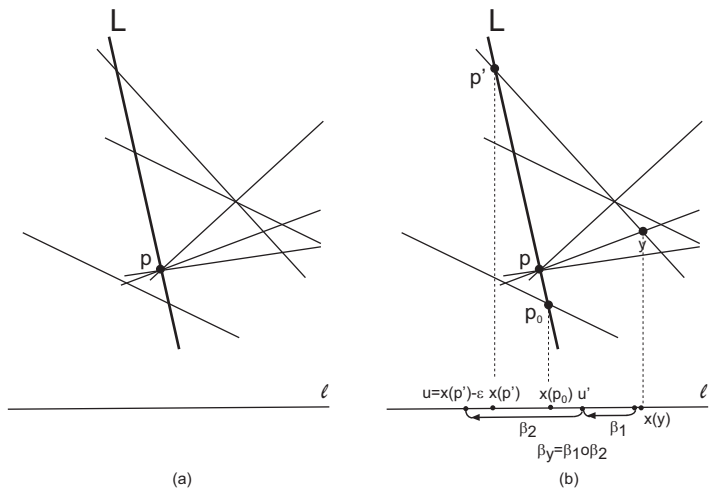}
\caption{Part (a) describes the setting for adding the line $L$ in the case of real line arrangements. Part (b) depicts the notations for the proof of the conjugation-free property for Step \ref{lemConjFreeLine}.}\label{addingLine3}
\end{figure}

We introduce a few notations.
Let $p'$ be the highest intersection point on $L$. Numerate the intersection points on $L$, starting from the lowest point on $L$, by $p_0,\,p_1,\,\ldots,\,p_v=p,\,\ldots,\,p_u=p'$ (note that all the points $p_i$ are nodes, except for $p_v$ which is an intersection point of multiplicity $m>2$) and denote their corresponding Lefschetz pairs by

 \begin{itemize}
  \item $s_0 = [1,2], s_1 = [2,3], \ldots, s_{v-1} = [v,v+1],$
  \item $s_v = [v+1,v+m]$,
  \item $s_{v+1}=[v+m,v+m+1],\ldots, s_u = [n-1,n]$,
\end{itemize}
 \noindent
 where $n$ is the number of lines in $\LL$. 

\medskip

We now divide the proof into a number of steps, each concentrating in a different domain, depending on the location of the basepoint.
First, we prove that for the basepoint $u \in \ell$, where $u = x(p')-\varepsilon$, $0 < \varepsilon \ll 1$ (see Figure \ref{addingLine3}(b)), $\pi_1(\CC^2 - \LL,u)$ has a conjugation-free geometric presentation.

\begin{step} For $u = x(p')-\varepsilon$, $\pi_1(\CC^2 - \LL,u)$ has a conjugation-free geometric presentation. \label{lemConjFreeLine}
\end{step}

\begin{proof}
We prove this step in two parts.
 In Part (I), we prove that all the skeletons of the intersection points to the left of $u$ or to the right of $u' \doteq x(p_0)+\varepsilon$ are exactly the same when computed in $\LL$ or in $\LL'$. Once we  prove that and we prove also that all the relations induced by the intersection points on $L$ have no conjugations, then we can use the fact that $\pi_1(\CC^2 - \LL',u)$ is conjugation-free for proving, in Part (II), that $\pi_1(\CC^2 - \LL,u)$
  has a conjugation-free geometric presentation.

\medskip
\noindent
\emph{\textbf{Part (I)}:}  Choosing the basepoint to be $u$ yields almost immediately that all the relations induced by intersection points on $L$ have no conjugations. As for the intersection points to the left of $u$,  since the intersection points on $L$ are not affecting any of these intersection points, we can consider the resulting skeletons
 as skeletons of $\LL'$ (where the only difference is that the fiber contains now an additional point, corresponding to the line $L$, being the highest point in the fiber and not participating in the skeletons). We now have to treat  the relations induced by the intersection points which are to the right of $L$.

For every intersection point $y \in \LL$ which is to the right of $u'$, let $s_y$ be the initial skeleton of $y$. We now  consider the braids which act on $s_y$ while going along a path $\beta_y$ starting at $x(y)-\varepsilon$ and ending at $u$. This path can be divided into two parts: the first part $\beta_1$ which starts at $x(y) - \varepsilon$ and ends at $u'$ and the second part $\beta_2$ which starts at $u'$ and ends at $u$ (see Figure \ref{addingLine3}(b)).

Let $(s'_y)_1$ (resp. $(s'_y)_2$) be the skeleton after the action of the braids induced by $\beta_1$ (resp. $\beta_y$) in the arrangement $\LL'$ and let
$(s_y)_1$ (resp. $(s_y)_2$) be this skeleton after the action of the braids induced by $\beta_1$ (resp. $\beta_y$) in the arrangement $\LL$.

Note that the fibers over $x(y) + \varepsilon$ in $\LL$ and  in $\LL'$ differ only by one point - in $\LL$ there is an additional point in the fiber, being the lowest point (numerated by $1$) and representing the line $L$. Since the braids induced by $\beta_1$ in the  arrangement $\LL$ do not induce any motion on this point in the fiber, numerated by $1$ (as the line $L$ does not intersect the other components in the section between $x(y)$ and $u'$), the skeletons $(s'_y)_1$  and $(s_y)_1$ are the same, except for that all the indices in $(s_y)_1$ are increased by $1$.

Note that the braid induced by the path $\beta_2$ in $\LL'$ is $$\Delta^{-1}\langle v+1,v+m-1 \rangle,$$ as $p \in \LL'$ is an intersection point with multiplicity $m-1$. On the other hand, the braid induced by the path $\beta_2$ in $\LL$ is:
$$
B_2 \doteq \Delta^{-1}\langle 1,2 \rangle  \cdots \Delta^{-1}\langle v,v+1 \rangle  \Delta^{-1}\langle v+1,v+m \rangle \Delta^{-1}\langle v+m,v+m+1 \rangle \cdots \Delta^{-1}\langle n-1,n \rangle.
$$

Now, when performing  $B_2$ on $(s_y)_1$, recall that $\beta_2$ is a path beneath all the intersection points on $L$ (that is, there are no intersection points on $L$ outside the section $\{ a \in \C^2 : u' < x(a) < u \}$). Since  $L$ is almost vertical with a negative slope, the braid $B_2$ induces the motion $\delta^{-1}_{n}$ on the point numerated by $1$  in the arrangement $\LL$ (in the fiber over $u'$), that is, the point $1$ performs a $180^\circ$ rotation clockwise, eventually turning into the point numerated by $n$. Moreover, as the point $1$ (in the fiber over $u'$) is not involved in the skeleton $(s_y)_1$, the only braid that does affect  $(s_y)_1$ is $\Delta^{-1} \langle v+1,v+m \rangle$. At the stage of applying $\Delta^{-1} \langle v+1,v+m \rangle$ (in $\LL$), the point  numbered as $1$ is then numbered as $v+1$ and 
the skeleton $(s_y)_1$ is passing beneath it. The braid $\Delta^{-1} \langle v+1,v+m \rangle$ sends the point $v+1$ to the point $v+m$ and the effect on the skeleton itself is as the effect of $\Delta^{-1} \langle v+1,v+m-1 \rangle$. This means that the skeletons $(s'_y)_2$ and $(s_y)_2$ are the same, but in the model of $(s_y)_2$ there will be one additional point corresponding to $L$, which will be the highest point.

\medskip
\noindent
\emph{\textbf{Part (II)}:} In this part we prove that $\pi_1(\CC^2 - \LL,u)$ has a conjugation-free geometric presentation, using the fact that $\pi_1(\CC^2 - \LL',u)$ has a conjugation-free presentation. Note that this is the same method used in the second part of \cite[Proposition 2.2]{EGT2}, and we bring it here for the
convenience of the reader.

We know that  $\pi_1(\CC^2 - \LL',u)$ has a conjugation-free geometric presentation;
hence we have that by a simplification process starting from the corresponding Zariski-van Kampen presentation, one can reach a presentation without conjugations.
If we imitate the simplification process of the presentation of $\pi_1(\CC^2 - \LL',u)$ for the presentation of  $\pi_1(\CC^2 - \LL,u)$,
the cases in which we need to use the relations induced by the point $p$ are the relations that have been simplified by using the relations induced by $p$ before adding the line $L$.
Recall that the  Lefschetz pair of $p$ in $\LL$ is $[v+1,v+m]$, and  in $\LL'$ the corresponding Lefschetz pair is $[v+1,v+m-1]$.
Denote by $\{\G_i\}$ the geometric generators both in $\pi_1(\CC^2 - \LL',u)$ and  $\pi_1(\CC^2 - \LL,u)$.

Thus,
in $\pi_1(\CC^2 - \LL',u)$, the relations induced by $p$ are:
$$
R_p: \quad [ \G_{v+1}, \G_{v+2},\ldots, \G_{v+m-1}]=e
$$

Since  in the fiber over $u$ in the arrangement $\LL$, the line $L$ is numbered as $n$ (and the numeration of the other lines was not changed), the relations induced by $p$ in $\pi_1(\CC^2 - \LL,u)$  are:
$$
\tilde{R}_p: \quad [ \G_{v+1}, \G_{v+2},\ldots, \G_{v+m-1},\G_n]=e
$$

We can divide the relations in $\pi_1(\CC^2 - \LL',u)$ into two subsets:
\begin{enumerate}
\item Relations that during the simplification process contain the subword
$$
\G^{-1}_{v+1} \cdots \G^{-1}_{v+m-2} \G_{v+m-1} \G_{v+m-2} \cdots \G_{v+1}.
$$

\item Relations that do not contain the above subword during its simplification process.
\end{enumerate}

For the second subset, the simplification process will be identical before adding the line $L$ and after it, since all the other relations induced by points of $\LL'$ have not been changed by adding the line $L$.

For the first subset, let  $R$ be a relation from this subset.  Except for applying the relations induced by $p$, the rest of the simplification process is identical to the one before adding the line. The only change is in the step of applying $R_p$. In this step, before adding the line $L$, the generator $\G_n$ has not been involved in $R_p$, but after adding the line $L$, it appears in $\tilde{R}_p$. Hence, for applying $\tilde{R}_p$, we have to conjugate the relation $R$ by $\G_n$, and using the commutative relations  which $\G_n$ is involved, we can diffuse $\G_n$ into the relation $R$, so we can use the relation $\tilde{R}_p$ instead of $R_p$.

Hence, we can simplify all the conjugations in all the relations in $\pi_1(\C^2 - \LL,u)$ as we simplify the relations in $\pi_1(\C^2 - \LL',u)$, so we get a conjugation-free geometric presentation, as needed. \end{proof}

\begin{step} For $u = x(p_0)+\varepsilon$, $\pi_1(\CC^2 - \LL,u)$ has a conjugation-free geometric presentation. \label{lemConjFreeLine2}
\end{step}

\begin{proof}
 The proof of this step is almost identical to the proof of Step \ref{lemConjFreeLine}. Again, all the relations induced by the intersection points on $L$ have no conjugations. As for the intersection points to the right of $u$,  since the intersection points on $L$ are not affecting any of these intersection points, we can consider the resulting skeletons
 as skeletons of $\LL'$  except for that now all the indices are increased by $1$ (as the line $L$ is numbered as $1$). We now have to treat  the relations induced by the intersection points which are to the left of $L$.

 However, the treatment is almost identical to the treatment of Step \ref{lemConjFreeLine}: for an intersection point $y$, which is to the left of $p'$,
we consider the two initial skeletons in the fiber over $x(y)+\varepsilon$ in $\LL'$ and in $\LL$, which are the same, except for that
in $\LL$ there is one additional point corresponding to $L$, which will be the highest point. Now, we consider  a path $\beta_y$ from $x(y)+\varepsilon$ to $u$ and we divide it into two parts: the first part $\beta_1$, from $x(y)+\varepsilon$ to $x(p')-\varepsilon$, and the second part $\beta_2$, from $x(p')-\varepsilon$ to $u$. As before, the first part does not induce any difference between the two above skeletons.
The second part $\beta_2$ causes eventually that the highest point in the fiber in $\LL$  performs a $180^\circ$ rotation counterclockwise,
while as for the rest of the skeleton, the composition of the braids corresponding to this part in $\LL$ induced the same operation as if it were in $\LL'$.
This means that the final skeletons induced by $y$ (the one in $\LL'$ and the second in $\LL$) are the same, except for that all the indices in the second are increased by $1$.
Now we can use the same simplification process with the method described in Part (II) of the proof of Step \ref{lemConjFreeLine}.
\end{proof}

Recall that $N$ denotes the set of the images of the singular points (with respect to the projection $\pi$) on $\ell$. We now prove that $\pi_1(\CC^2 - \LL,u)$ has a conjugation-free geometric presentation for every basepoint $u$ that belongs to  one of the following domains:

\begin{itemize}
\item $D_1 = \{ u \in \ell - N : x(p') < u < x(p_0) \}$.
\item $D_2 = \{ u \in \ell - N : x(p_0) < u \}$.
\item $D_3 = \{ u \in \ell - N : x(p') > u \}$.
\end{itemize}

\medskip
\begin{step} For $u \in D_1$, $\pi_1(\CC^2 - \LL,u)$ has a conjugation-free geometric presentation. \label{lemConjFreeLine3}
\end{step}

\begin{proof}
If $u \in D_1$,  we can use Remark \ref{remG3_Action}(II), stating that when we pass $u$ ``below" the image of a node (with respect to the projection), then the conjugation-free property is preserved. Then, $\pi_1(\C^2-\LL,u)$ is conjugation-free, since we can either start from the basepoint $x(p')-\varepsilon$ (as we know by Step \ref{lemConjFreeLine} that $\pi_1(\CC^2 - \LL, x(p')+\varepsilon)$ is conjugation-free)  and ``pass" $u$ to the right, going only below nodes till we almost reach $x(p)$; or we can start from the basepoint $x(p_0)-\varepsilon$ (as we know by Step \ref{lemConjFreeLine2} that $\pi_1(\CC^2 - \LL, x(p_0)+\varepsilon)$ is conjugation-free)  and ``pass" $u$ to the left, going only below nodes till we almost reach $x(p)$ (from its other side).
\end{proof}

\begin{step} For $u \in D_2$, $\pi_1(\CC^2 - \LL,u)$ has a conjugation-free geometric presentation. \label{lemConjFreeLine4}
\end{step}

\begin{proof}
The proof of this step is based on proving the following three claims:
\begin{enumerate}
\item  All the skeletons induced by the intersection points not on $L$ are the same in $\LL$, as if they were computed in $\LL'$,
except for that all the corresponding indices in $\LL$ are increased by $1$.
\item All the relations induced by the nodes on $L$ have no conjugations.
\item The relation induced by $p \in L$ has no conjugations.
\end{enumerate}

If we prove these claims, then we are done, as we can use the same simplification process done for the skeletons associated to the intersection points not on $L$
in $\pi_1(\CC^2 - \LL',u)$, as we have done in Part (II) in the proof of Step \ref{lemConjFreeLine}.

As for claim (1), the proof is essentially the same as in Step \ref{lemConjFreeLine2}.
We now prove claim (2) by induction on the location of the basepoint $u$ with respect to the intersection points of $\LL$ to the right of $p_0$. Let us numerate the intersection points of $\LL$ in the domain
$\{m \in \CC^2 : x(m) \geq x(p_0)\}$ by
$q_0=p_0,q_1,\ldots,q_k$, where $x(q_i)<x(q_j)$ for $i<j$. Now, as was proved in Step \ref{lemConjFreeLine2}, the claim is correct
when $u \in \{ m \in \ell : x(q_0) < x(m) < x(q_1) \}$ which is the base of the induction. Assume that the claim is correct for all $u \in \{ m \in \ell : x(q_0) <  x(m) < x(q_i) \}$. We now prove that the claim is correct when $u$ is in the next domain $\{ m \in \ell : x(q_i) < x(m) < x(q_{i+1}) \}$.

Assume that $u = x(q_i)-\varepsilon$, and that in the fiber over $u$, the local numeration of the lines intersecting in $q_i$ is $a,a+1,\ldots,a+s=b$, where $1<a$ (as the line $L$ is always numbered as $1$ in the domain $\{ m \in \ell :  x(m) > x(p_0) \}$). By Remark \ref{remG3_Action}(I), when we move $u$ to $u'=x(p_i)+\varepsilon$, the induced isomorphism  affects only the geometric generators
$\G_a,\ldots,\G_b$ by a Hurwitz move. Explicitly, the isomorphism is:

 \begin{equation} \label{eqnf3'}
f'_3(\Ga _i) =
\left\{ \begin{array}{ccl}
{\G_i} && {i<a \text{ or } i>b}  \\
{{\Ga' _b}  \cdots {\Ga'_{a+1}} \Ga' _{a} {\Ga'^{-1}_{a+1}} \cdots  {\Ga_b'}^{-1}}  && {i=a+s=b} \\
{{\Ga' _b} \cdots {\Ga'_{a+2}} \Ga_{a+1}' {\Ga'^{-1}_{a+2}} \cdots {\Ga _b'}^{-1}} && {i=b-1} \\
{\vdots} && {\vdots} \\
{{\Ga' _b}  \Ga_{b-1}' {\Ga _b'}^{-1}} && {i=a+1} \\
{\Ga _b' } && {i=a,}
\end{array}
\right.
\end{equation}

\noindent
where $\{\G'_i\}$ are the geometric generators of $\pi_1(\CC^2 - \LL,u')$.
Note that the isomorphism $f_3$ described in Equation (\ref{eqnIsoF3}) is associated to moving a basepoint from right to left, while
$f_3'$ is associated to moving a basepoint from left to right; explicitly, $f_3' = f_3^{-1}$.

We split the proof into two cases.

\medskip \noindent
\textbf{Case (1)}: There is no line passing through $p$ and $q_i$.

When the basepoint is $u$, then the  relations induced
by the nodes on $L$ are:
$$
[\G_1,\G_j]=e,\, a\leq j \leq b
$$
(by the conjugation-free property).
Thus, the above isomorphism $f_3'$ affects these commutative relations in the following way:

$$[\G_1,\G_a]=e \,\, \to \,\, [\G'_1,\G'_b]=e,$$
$$[\G_1,\G_{a+1}]=e \,\, \to \,\, [\G'_1,\G'_b\G'_{b-1}{\G'_b}^{-1}]
\overset{[\G'_1,\G'_b]=e}{=}
[\G'_1,\G'_{b-1}]=e,$$
$$[\G_1,\G_{a+2}]=e \to [\G'_1,\G'_b\G'_{b-1}\G'_{b-2}{\G'^{-1}_{b-1}}{\G'_b}^{-1}]
\overset{[\G'_1,\G'_b]=[\G'_1,\G'_{b-1}]=e}{=}
[\G'_1,\G'_{b-2}]=e,$$

 \noindent
and we continue in the same way for the other relations in order to get rid of the conjugations.

If there are nodes on $L$ such that the lines passing through them are numbered as
  $1$ and $m$ (when numerated locally over $u'$), when $m<a$ or $m>b$,
then the  relation $[\G_1,\G_m]=e$ in $\pi_1(\CC^2 - \LL,u)$  is left unchanged in $\pi_1(\CC^2 - \LL,u')$, i.e. the corresponding relation  still has no conjugations.

\medskip \noindent
\textbf{Case (2)}: There is a line  passing  through $p$ and $q_i$.

Assume that the line
passing through $p$ and $q_i$ is numbered, locally over $u$, as $a+k$. Therefore, when the basepoint is $u$, the
relations induced by the nodes on $L$  are:
$$[\G_1, \G_x]=e, \, a  \leq x \leq b, x \neq a+k.$$
After the isomorphism, we know in any case that the
relations:

\begin{equation}\label{eqnRelPres}
[\G'_a,\G'_{a+1},\G'_{a+2},\ldots,\G'_b]=e
\end{equation}
exist (these are the relations induced by $q_i$).
 The treatment in the relations $$[\G_1, \G_x]=e, a  \leq x \leq a+k-1$$ is exactly as in case (1).
 We now treat the remaining relations, i.e. the relations $$[\G_1, \G_x]=e,\,\, a+k+1  \leq x \leq b.$$ Now,
{\small{  \begin{equation}\label{eqnRelPres2}[\G_1,\G_b]=e \to [\G'_1,\G'_b\G'_{b-1}\cdots \G'_{a+1}\G'_a{\G'^{-1}_{a+1}}\cdots{\G'^{-1}_{b-1}}{\G'_b}^{-1}] \overset{\rm{Eqn.}\, (\ref{eqnRelPres})}{=}
[\G'_1,\G'_a]=e.\end{equation}

  \begin{equation}\label{eqnRelPres3} [\G_1,\G_{b-1}]=e \to [\G'_1,\G'_b\G'_{b-1} \cdots \G'_{a+2}\G'_{a+1}{\G'^{-1}_{a+2}}\cdots{\G'^{-1}_{b-1}}{\G'_b}^{-1}]
  \end{equation}
  $$ = [\G'_a\G'_1{\G'_a}^{-1},\G'_a\G'_b\G'_{b-1} \cdots \G'_{a+2}\G'_{a+1}{\G'^{-1}_{a+2}}\cdots{\G'^{-1}_{b-1}}{\G'_b}^{-1}{\G'_a}^{-1}]
   \overset{\rm{Eqns.}\, (\ref{eqnRelPres}),(\ref{eqnRelPres2})}{=}
[\G'_1,\G'_{a+1}]=e.$$

$$
 [\G_1,\G_{b-2}]=e \to [\G'_1,\G'_b\G'_{b-1} \cdots \G'_{a+3}\G'_{a+2}{\G'^{-1}_{a+3}}\cdots{\G'^{-1}_{b-1}}{\G'_b}^{-1}] =
$$
$$
 = [\G'_a\G'_{a+1}\G'_1{\G'^{-1}_{a+1}}{\G'_a}^{-1}, \G'_a\G'_{a+1}\G'_b\G'_{b-1} \cdots \G'_{a+3}\G'_{a+2}{\G'^{-1}_{a+3}}\cdots{\G'^{-1}_{b-1}}{\G'_b}^{-1}{\G'^{-1}_{a+1}}{\G'_a}^{-1}] = $$$$
   \overset{\rm{Eqns.}\, (\ref{eqnRelPres})-(\ref{eqnRelPres3})}{=}
[\G'_1,\G'_{a+2}]=e,
$$
}}

\noindent
and we continue in the same way for the relations induced by the other nodes.

As in the previous case, if there are nodes on $L$ such that the lines passing through them are
(when numerated locally over $u$)  $1$ and $m$, when $m<a$ or $m>b$,
then the  relation $[\G_1,\G_m]=e$ in $\pi_1(\CC^2 - \LL,u)$  is left unchanged in $\pi_1(\CC^2 - \LL,u')$, i.e. the corresponding relation still has no conjugations.

\medskip

We now prove claim (3), that  the relations induced by $p \in L$ have no conjugations. Recall that the point $p$ is an intersection point of multiplicity $m>2$. The proof is similar to the proof of claim (2), i.e., it is by induction on the location of the basepoint $u$ with respect to the intersection points of $\LL$ to the right of $p_0$. Again, we know that this claim holds when  $u \in \{ m \in \ell : x(q_0) < x(m) < x(q_1) \}$, by Step \ref{lemConjFreeLine2}, which is the base of the induction. Assume that the claim is correct for all $u \in \{ m \in \ell : x(q_0) <  x(m) < x(q_i) \}$. We now prove that it is correct when $u$ is in the domain $\{ m \in \ell : x(q_i) < x(m) < x(q_{i+1}) \}$.

Assume that $u = x(q_i)-\varepsilon$, and denote $u' = x(q_i)+\varepsilon$. Let $\beta_u$ be  a path starting at $x(p)+\varepsilon$ and ending at $u$ (see Figure \ref{pathBu}). Denote by  $B_u$ the braid that is induced by the path $\beta_u$  in $\LL$ (i.e.  the Lefschetz diffeomorphism $\psi_{\beta_u}$; see Definition \ref{defLefdif}).

\begin{figure}[!ht]
\epsfysize 6.5cm
\epsfbox{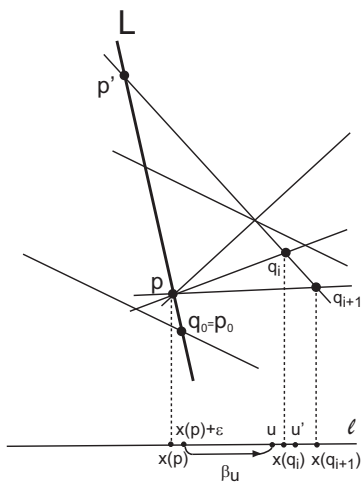}
\caption{Notations for the proof of claim (3) of Step \ref{lemConjFreeLine4}.}\label{pathBu}
\end{figure}

 Assume that in the fiber over $u$, the local numeration of the lines intersecting in $q_i$ is $a,a+1,\ldots,a+s=b$, where $1<a$ (as the line $L$ is always numbered as $1$ in the domain $\{ m \in \ell :  x(m) > x(p_0) \}$). Recall that the
 Lefschetz pair of the point $p$ in $\LL$  over $x(p)+\varepsilon$ is $[v+1,v+m]$.

When the basepoint is $u$, the simplified relation induced
by the point  $p$ is:
\begin{equation} \label{relMultiPoint}
[\G_1,\G_{B_u(v+2)},\G_{B_u(v+3)},\ldots,\G_{B_u(v+m)}]=e,
\end{equation}
by the conjugation-free property of $\pi_1(\CC^2 - \LL,u)$, which is the induction hypothesis (the notation $B_u(x)$ refers to the numeration of the point $x$ after applying the braid $B_u$ on it, when considering the braid group as modeled by the group of diffeomorphisms on a punctured disk, where $x$ is one of the punctures; see Definition  \ref{defBraidGr}).
Note that when passing to the basepoint $u'$, the induced isomorphism is $f_3'$ as described earlier (see Equation (\ref{eqnf3'})).
However, note that $f_3'$ affects  at most  one of the generators $\G_{B_u(x)},\,$ where $ v+2 \leq x \leq v+m$, since there
 is at most one line passing through both $p$ and $q_i$.
We now divide the proof into two cases: when there is no  line passing through $p$ and $q_i$ and when there is one.

\medskip \noindent
\textbf{Case (1)}: If there is no line passing through $p$ and $q_i$, then the isomorphism $f_3'$ does not affect
 relation (\ref{relMultiPoint}), i.e. all the generators $\G_i$ that participate in the relation are sent to $\G_i'$. This means that the relations induced by $p$ have no conjugations in $\pi_1(\CC^2 - \LL,u')$.

\medskip \noindent
\textbf{Case (2)}: There is a line passing through $p$ and $q_i$. Assume that  this line is numbered over $u$ as $B_u(j) = a+t$ (where $0 \leq t \leq s$, $v+2 \leq j \leq v+m$), that is, only the generator $\G_{a+t}$ is affected by $f_3'$ and the other generators $\G_{B_u(x)},\, v+2 \leq x \leq v+m, x \neq j $, are sent to $\G'_{B_u(x)}$. Now,
$$
f_3'(\G_{B_u(j)}) = f_3'(\G_{a+t}) \doteq {\Ga' _b} \cdots {\Ga'_{b-t+1}} \Ga_{b-t}' {\Ga'^{-1}_{b-t+1}} \cdots {\Ga _b'}^{-1},
$$

Note that over $u'$, the line, that passes through $p$ and $q_i$, is numbered as $b-t$. This means that the lines numbered as
$b,b-1,\ldots,b~-~t~+~1$  over $u'$, do not pass through $p$ and therefore they intersect  the line $L$ in nodes. By claim (2), we have
in $\pi_1(\CC^2 - \LL,u')$ that $[\G_1',\G_x']=e,\, b-t+1 \leq x \leq b$.

Thus, the relation induced by $p$ in $\pi_1(\CC^2 - \LL,u')$ is now:

$$
\mathcal{\breve{R}}_p \, : \, [\G_1',\G'_{B_u(v+2)},\ldots, \G'_{B_u(j-1)}, f_3'(\G_{B_u(j)}), \G'_{B_u(j+1)}, \ldots ,\G'_{B_u(v+m)}]=e.
$$

We claim that we can perform the  same simplification process, done in $\pi_1(\C^2-\LL',u')$, for the relation $\mathcal{R}_p$ induced by $p$, after moving from $u$ to $u'$. Indeed, the relation $\mathcal{R}_p$ is

$$
\mathcal{R}_p \, : \, [\G'_{B_u(v+1)},\ldots, \G'_{B_u(j-2)}, f_3'(\G_{B_u(j-1)}), \G'_{B_u(j)},\ldots\G'_{B_u(v+m-1)}]=e
$$
(recall that we have to increase the indices by $1$ when passing from $\LL'$ to $\LL$ when the basepoint is to the right of $x(p_0)$).

   We have to consider two cases: the first is that during the simplification process
  of $\mathcal{R}_p$, we are simplifying $f_3'(\G_{B_u(j-1)})$; and the second is that during the simplification process
  of $\mathcal{R}_p$, we are simplifying a product of $f_3'(\G_{B_u(j-1)}) \cdot \G'_{B_u(j-2)}$, which might be turned, in $\mathcal{\breve{R}}_p$, to a product of $f_3'(\G_{B_u(j)}) \cdot \G'_1 \cdot \G'_{B_u(j-1)}$. Now, only the second case can cause problems (as the first case does not involve $\G_1'$). However,  note that  in the product $f_3'(\G_{B_u(j)}) \cdot \G'_1$, we can diffuse $\G'_1$ till it reaches the generator $\Ga_{b-t}'$ (since $[\G_1',\G_x']=e,$ for all $ b-t+1 \leq x \leq b$ as was explained above) and now apply the same simplification
  process as was done in $\pi_1(\C^2-\LL',u')$.

Therefore,  the relation induced by $p$  has no conjugations and now we can use  the method described in Part (II) in the proof of Step \ref{lemConjFreeLine}. \end{proof}

\begin{step} For $u \in D_3$, $\pi_1(\CC^2 - \LL,u)$ has a conjugation-free geometric presentation. \label{lemConjFreeLine5}
\end{step}

\begin{proof}
The proof of this step is similar to the proof of Step \ref{lemConjFreeLine4}. The only difference is that now we use the isomorphism
$f_3$ (and not $f_3'$), as we move (during the inductive step of the proofs of claims (2) and (3)) from the basepoint $x(q_i)+\varepsilon$ to the basepoint $x(q_i)-\varepsilon$. However, as can be easily seen, we can use the same techniques for proving these claims. \end{proof}

The above series of steps completes the proof of Proposition \ref{lemAddLineComProof}(1). \end{proof}

\section{Preservation of Conjugation-Freeness: The case of real CL arrangements}\label{secRealCLArr}

The goal of this section is to prove Proposition \ref{lemAddLineComProof}(2): i.e.,  under certain conditions, if we add a line $L$
to a CL arrangement $\A$ whose fundamental group is  conjugation-free, such that $L$ passes through at most one singular point of $\A$, then the fundamental group of $\A \cup L$
is also conjugation-free. The additional condition, that is not required in the case of real line arrangements, is that $\beta(\A \cup L)=0$.
However, we will see in Section \ref{secOneCycleNotPass} that this condition can be weakened.
Moreover,
in the course of the proof, we prove other important results for CL arrangements, such as the commutation of the generator associated to the conic with all the other geometric generators (associated to the lines)
in $\pi_1(\C^2 - \A)$.

\medskip

When trying to follow the proof given in the previous section (for real line arrangements),
  in addition to the operations $\h_1,\h_2,\h_3$, which can also be operated on a real CL arrangement, we have to consider the following new operation:

 \begin{itemize}
\item $\h_4 $: If a real rotation of a line causes a tangency point between this line and the conic, perform a rotation in $\CC^2$ over the possible tangency point;
\end{itemize}
\noindent
Moreover, we have to reconsider the following operations:

 \begin{itemize}
\item $\h_1 $: Rotating a line over a multiple point, as in the case of real line arrangements.
\item $\h_3 $: Changing the basepoint $u \in \ell -N$, now when  moving $u$ from one side of a branch point to its other side.
\end{itemize}

We note that the proof of the preservation of the conjugation-free property of the operation $\h_2$ (rotating a line over another line) holds in tact also for the CL arrangements specified in
Proposition \ref{lemAddLineComProof}(2), as it involves only the generators of the lines.
The proof of the preservation of the conjugation-free property of the operation $\h_3$ will have to be changed, as we will now deal
 with complex points too.


\begin{lemma} \label{g1_op_CFpreserve_CL}
For a real CL arrangement $\A$ with one conic, the operation $\h_1$ preserves the conjugation-free property.
\end{lemma}

\begin{proof}
Let $b_1,b_2$ be the two branch points of the conic with respect to the projection.
We have to consider two cases: the first is  that during the rotation of the line $L \in \A$, $L$ does not pass through one of the $b_i$'s. The second is that it does pass through one of them.

As for the  first case, the proof is the same as the proof of Lemma \ref{g1_op_CFpreserve}. As for the second case, we have the situation depicted in Figure \ref{rotateOverBranch}(a).

\begin{figure}[!ht]
\epsfysize 7cm
\epsfbox{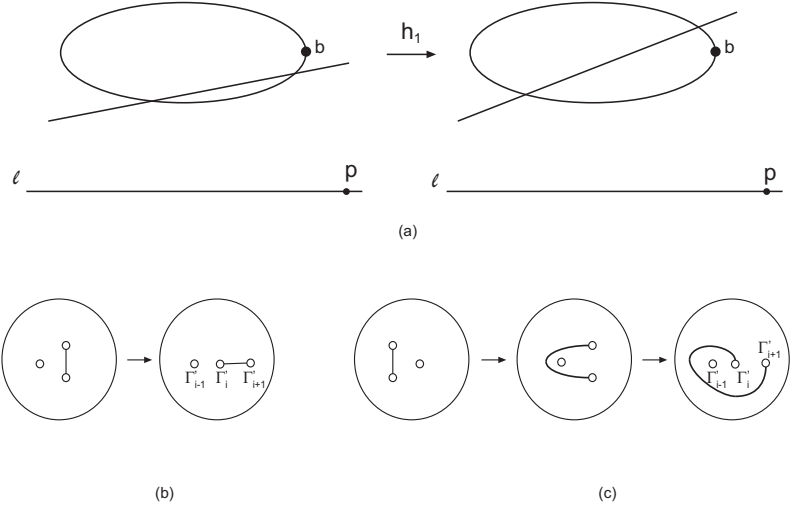}
\caption{Part (a) is a rotation of a line over a branch point $b$ and parts (b) and (c) are its effects on the skeleton associated to the branch point $b$.}\label{rotateOverBranch}
\end{figure}

As before, the isomorphism $\pi_1(\C^2 - \A,u) \to \pi_1(\C^2 - \h_1(\A),u)$ is $\G_i \mapsto \G'_i$.
Thus, the only relation that might be changed by this rotation is the one induced by the branch point $b$: after considering the complex points of the fiber
as real points, the relation might be transposed to the following: $\{\G_i = \G_{i+1}\} \mapsto \{ \G'_{i-1}\G'_i{\G'^{-1}_{i-1}} = \G'_{i+1} \}$,
where $\G'_{i-1}$ is the generator of the line and $\G'_i, \G'_{i+1}$ are the generators of the conic in $\pi_1(\C^2 - h_1(\A),u)$ (see Figures \ref{rotateOverBranch}(b) and \ref{rotateOverBranch}(c)). However, due to the isomorphism above, we still have the relation
$[\G'_{i-1}, \G'_{i}]=e$ (or $[\G'_{i-1}, \G'_{i+1}]=e$; note that $\pi_1(\C^2 - \A)$ is conjugation-free) and thus the resulting relation induced by the branch point  has no conjugations. So, we get that the presentation of $\pi_1(\C^2 - \h_1(\A))$ is also conjugation-free.
\end{proof}

\begin{lemma} \label{g5_op_CFpreserve_CL}
For a real CL arrangement $\A$ with one conic, the operation $\h_4$ preserves the conjugation-free property.
\end{lemma}

\begin{proof}

Given a line $L \in \A$ which we have to rotate, the rotation of $L$ in $\RR^2$ might not be an equisingular deformation, as at a particular point during the rotation, $L$ may be tangent to the conic.
In this case, we  change the rotation in $\RR^2$ to the following equisingular deformation, according to the  local model appearing  in Figure \ref{localModel}(a).

\begin{figure}[!ht]
\epsfysize 3cm
\epsfbox{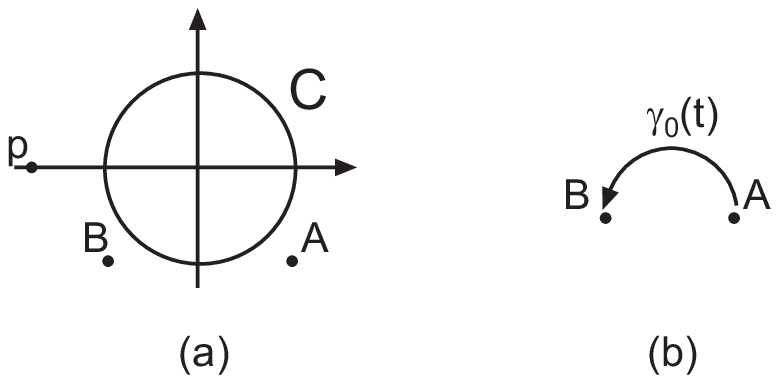}
\caption{Part (a): a local model for a complex rotation: $$C = \{(x,y)\in \C^2 : x^2+y^2=1\},\ p=(-2,0).$$ Part (b) illustrates the path $\gamma_0(t)$ used to rotate the line.}\label{localModel}
\end{figure}

Let $C = \{(x,y)\in \C^2 : x^2+y^2=1\},\ p=(-2,0)$, $A = (1,-1)$ and $ B = (-1,-1)$. Let $\gamma_0:[0,1] \to \C^2$ be half a circle of radius $1$ on the plane $\{y=-1\}$ in $\C^2$ which starts at $A$ and ends at $B$ (see Figure \ref{localModel}(b)). Let $L(t)$ be the family of lines passing through $p$ and $\gamma_0(t)$, for $0 \leq t \leq 1$. Obviously, $C \cup L(t)$ is an equisingular deformation and thus $\pi_1(\C^2 - (C \cup L(0))) \cong \pi_1(\C^2 - (C \cup L(1)))$. Denote by $h_4$ this operation after which $L = L(1)$.

We  have to prove that if $\pi_1(\C^2 - \A,u)$ is  conjugation-free, then also $\pi_1(\C^2 - \h_4(\A),u)$ is  conjugation-free.
We have to check, for all the initial skeletons associated to the singular points to the right of $p_1,p_2$, that
the braids that are applied on them, while passing below $x(p_1),x(p_2)$, along a path $\beta$, are the same in the CL arrangements $\A$ and $\h_4(\A)$; see the models depicted in Figures \ref{rotateOverTangent}(a), and  \ref{rotateOverTangent}(b).

\begin{figure}[!ht]
\epsfysize 4cm
\epsfbox{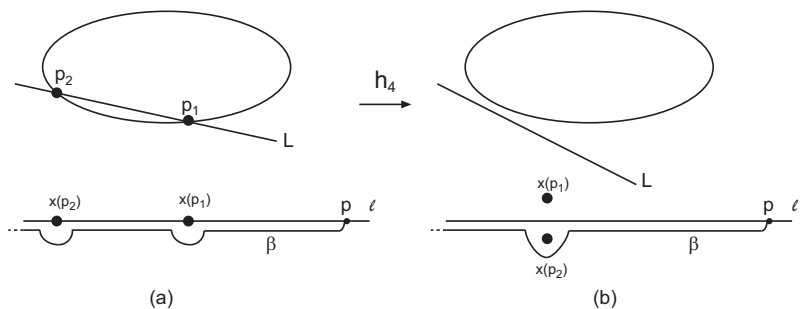}
\caption{A local model for a complex rotation and its effect on the path $\beta$.}\label{rotateOverTangent}
\end{figure}

Before the rotation of $L$ (see Figure \ref{rotateOverTangent}(a)), the braids that are applied on the skeletons while passing below $x(p_1)$ and $x(p_2)$, along a path $\beta$, are $\Delta\langle t,t+1 \rangle\cdot \Delta\langle t,t+1 \rangle = (H(\sigma))^2$, where $t$ is the local index of the
point in the fiber which corresponds to the line $L$, $t+1$ is the local index of the ``lower" generator of the conic and $\sigma$ is the segment connecting these two points in the fiber. In order to analyze the situation after the rotation of $L$ (see Figure \ref{rotateOverTangent}(b)), recall the following setting from \cite{FT5point}:
let $C_0=\{(y^2-x)(y+x+1)=0\}$ (see its real part in Figure \ref{ComplexInter}(a)). Define the projections $\pi_1,\pi_2:C_0 \to\mathbb{C}$ by $\pi_1(x,y) = {x},\,\pi_2(x,y) = {y}$. Denote by $p_1,
p_2$ the complex intersection points of the curves $y^2=x$ and $y=-x-1$. Define:
$$x_0=-\frac{1}{4},\, A=\pi_2(\pi_1^{-1}(x_0))=\left\{\pm{\frac{1}{2}i},-\frac{3}{4}\right\},$$
$$x_1=-\frac{3}{4},\,A'=\pi_2(\pi_1^{-1}(x_1))=\left\{\pm\frac{\sqrt{3}}{2}i,-{\frac{1}{4}}\right\}.$$
Let $\gamma(t), 0\leq
t \leq 1,$ be a curve starting at $x_1$, ending at $x_0$ and
surrounding $x(p_2)$ from below (see Figure \ref{ComplexInter}(b)). Let
$D$ be a disk in the $y$-axis such that $A,A'\subset D$ and denote by $\psi_{\gamma}$ the Lefschetz diffeomorphism induced by $\gamma(t)$. Let
$\sigma$ be the segment connecting $-\frac{1}{4}$ and
$\frac{\sqrt{3}i}{2}$ in $A'$, see Figure \ref{ComplexInter}(c). Then, by \cite[Corollary 2.2]{FT5point}, $\psi_{\gamma}=
(H(\sigma))^2$, where $H(\sigma)$ is the half-twist induced by the
path $\sigma$.

\begin{figure}[!ht]
\epsfysize 4cm
\epsfbox{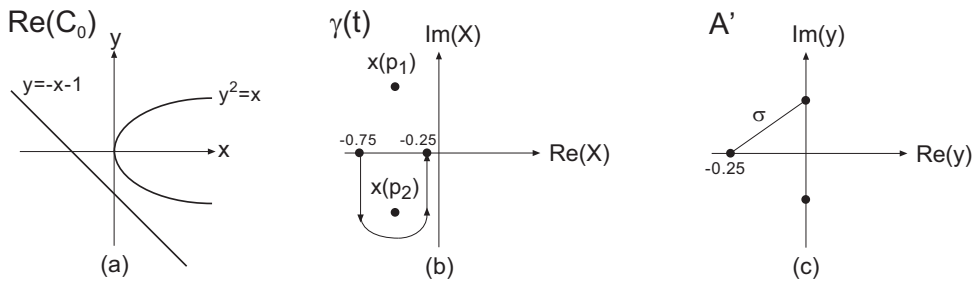}
\caption{A local model for a complex intersection between a line and a conic.}\label{ComplexInter}
\end{figure}

It is now clear that if we rotate the two complex points by  $90^{\circ}$  counterclockwise  in the fiber  (in order to obtain a
real model for the fiber), we get, in both cases, that the same braid  is applied, as needed.
\end{proof}

After we have proved these two lemmata, we are ready to finish the proof of Proposition \ref{lemAddLineComProof}(2).

\begin{lemma} \label{g3_op_CFpreserve_CL}
Let $\A$ be a real CL arrangement with one conic such that $\beta(\A)=0$. Assume that $\A$ can be presented as a union  $\A = \A' \cup L$, where $\A'$ is a real CL arrangement
with a conjugation-free presentation of $\pi_1(\C^2 - \A',u)$
for \emph{real every} basepoint $u \in \ell-N$, and $L \not\in \A'$ is a real line that passes through at most  one singular point $p$ of $\A'$.
 Then, in this class of real CL arrangements,
the operation $\h_3$ preserves the conjugation-free property for $\A$, i.e. $\pi_1(\C^2 - \A,u)$ has a conjugation-free geometric presentation for every real basepoint $u \in \ell$.
\end{lemma}

As in Remark \ref{rem_NotRest},  note that this lemma is not so restrictive as it  seems in a first glance. Taking a real CL arrangement with one conic (ellipse or hyperbola) and some lines such that the only singular points (with respect to the projection) are nodes and the branch points
of the conic, the affine fundamental group is abelian and as the relations induced by the branch points are equalities between generators, this group has a conjugation-free presentation. Thus, according to the lemma, one can add lines passing through at most one singular point and get the desired real CL arrangement (see also  Section \ref{secCFGraphs} for more details).

\begin{proof}
 Note that we can  rotate the line $L$ around the point $p$ until it is almost vertical to the  reference line $\ell$ by operations $h_1$ and $h_2$.  We remind again that whenever we refer to the basepoint $u$, we implicitly assume that $u$ is real. We divide the proof to a number of steps.

  In the first step, we  prove this lemma for only two basepoints $u \in \ell - N$ (where $N$ is the set of the projections of the singular points). In the other steps, we emphasize the changes we have to make in order to modify the proof of Proposition \ref{lemAddLineComProof}(1), so that we can use the corresponding steps as in the proof of that proposition.

We prove this lemma for the case where the conic is either an ellipse or a hyperbola (recall that we do not include  a parabola in our real CL arrangements, see Remark \ref{remGenParab}(2)).

Let $p'$ (resp. $\al$) be the singular point on $L$ with the minimal (resp. maximal) real value of its $x$-coordinate. Note that here
$\al$ is not necessarily a real point (i.e. its coordinates can be complex), as can be seen in case (1) in the proof of Step \ref{step1CL}.

\begin{step} \label{step1CL}
 Lemma \ref{g3_op_CFpreserve_CL} holds for $u = x(\al)+\varepsilon$ or $u = x(p')-\varepsilon$.
\end{step}

\begin{proof}
We divide the proof into two cases. The first case covers the cases where either $C$ is an ellipse and $p$
 is to the left of its two branch points, or  $C$ is a hyperbola and $p$ is between the two branch points. The second case covers the additional two cases.
\medskip
\noindent

\medskip
\noindent
\textbf{Case (1):} This case deals with the case where either $C$ is a hyperbola and $p$ is between the two branch points, or $C$ is an ellipse and $p$
 is to the left of its two branch points (if $p$ is to the right of the two branch points, we can reflect  the arrangement with respect to the $y$-axis and proceed as follows). We can rotate the line $L$ to be a line with a very negative slope, almost vertical to the reference line $\ell$. Thus we get the arrangement in Figure \ref{pointMiddleLineCL_hyper}(a) or the arrangement in Figure \ref{pointMiddleLineCL_hyper}(b). We will work with the arrangement
 in Figure \ref{pointMiddleLineCL_hyper}(a) for notational reasons. However, the only thing that matters is that the line $L$ intersects the conic $C$ in two \emph{complex} intersection points $\alpha,\alpha'$.

\begin{figure}[!ht]
\epsfysize 8cm
\epsfbox{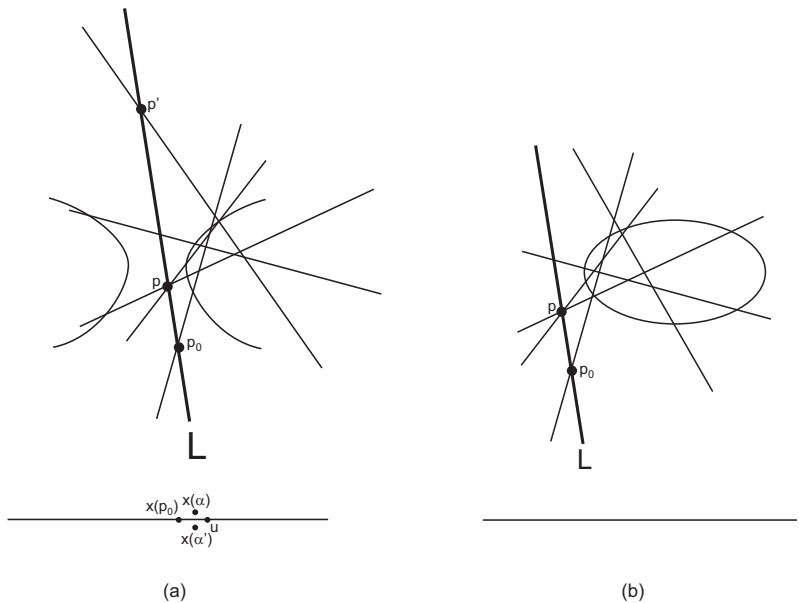}
\caption{First case: an illustration for the real part of $\A  = \A' \cup L$  with a hyperbola (part (a)) or with an ellipse (part (b)).}\label{pointMiddleLineCL_hyper}
\end{figure}


 As proved in Lemma \ref{g5_op_CFpreserve_CL}, the rotation of the line $L$ preserves the conjugation-free property (if it exists  before or after the rotation; recall that Lemma \ref{g5_op_CFpreserve_CL} has a local nature, i.e. it does not matter if $C$ is either an ellipse or a hyperbola). Let $p_0$ be the lowest real singular point on the line  $L$ (with respect to the $y$-coordinate). As the slope of $L$ is very negative, we can assume that $x(p_0) < x(\alpha)$ (by computation) and we choose a basepoint $u \in \ell$ such that  $u = x(\alpha)+\varepsilon$. We first prove  that the induced relations of all the singular points on $L$ have no conjugations.

Note that the relations in $\pi_1(\C^2 - \LL,u)$ induced by the complex intersection points $\alpha,\alpha'$ have no conjugations. Indeed,
one can get conjugations if the path,
going from the basepoint to the singular point in question, goes below other singular points. This is not the case here, as the points $\alpha,\alpha'$ are the closest points to $u$.

 Hence, the generator corresponding to the line $L$ commutes with the generators of the conic $C$. In the fiber over $u$ (i.e. the fiber $\CC^1_u = \pi^{-1}(u)$) there are two complex point $y(\beta), y(\beta')$, where $C \cap \pi^{-1}(u) = \{ \beta, \beta' \}$. We
consider the model of this fiber  as a real fiber by rotating  these complex points by $90^{\circ}$ counterclockwise; see Figure \ref{pointMiddleLineCL_hyper3}(a).

\begin{figure}[!ht]
\epsfysize 8cm
\epsfbox{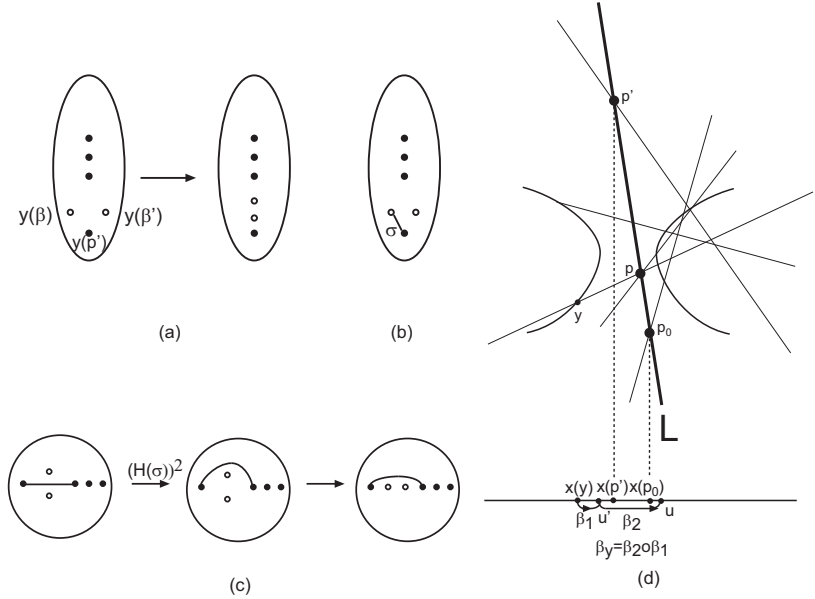}
\caption{Notations for computing the  relations induced by the singular points on $L$ (parts (a),(b),(c)) and for the  proof that the relations induced by the singular points to the left of $L$ have no conjugations (part (d)).}\label{pointMiddleLineCL_hyper3}
\end{figure}

As for the relation induced by the point $p_0$, we have to apply  the braid $(H(\sigma))^2$ on the initial skeleton of $p_0$, where $\sigma$ is drawn in Figure \ref{pointMiddleLineCL_hyper3}(b) (see the proof of Lemma \ref{g5_op_CFpreserve_CL}), and then
rotating the complex points by $90^{\circ}$ counterclockwise. The resulting skeleton is presented in Figure \ref{pointMiddleLineCL_hyper3}(c) and since the generator corresponding to $L$ commutes with the generators corresponding to the hyperbola, the induced relation has no conjugations. A similar computation works for all the other intersection points on $L$.
Thus, we proved that all the intersection points on $L$ induce relations without conjugations.

Note that  the singular points which are to the right of $u$  already appeared in $\A'$ and thus their associated skeletons are the same (when increasing the corresponding indices by 1). Let $p'$ be the highest intersection point on $L$.

We now prove that all the skeletons of the singular
points to the left of $p'$ are the same when computed in $\A$ and in $\A'$, besides the fact that  the indices of the points of the skeletons in $\A$ are increased by $1$ (this is due to the fact that $L$ is numbered as $1$ in the fiber over $u$). Once we  prove that, we can use the fact that $\pi_1(\CC^2 - \A',u)$ is conjugation-free and the methods outlined in part (II)  of the proof of Step \ref{lemConjFreeLine} to simplify these relations and to prove that $\pi_1(\CC^2 - \A,u)$  has a conjugation-free geometric presentation.

In order to prove that, we need additional notations. Numerate the intersection points on $L$,
 starting from the highest real intersection point on $L$, by $p_u = p',\ldots,p_v=p,\ldots,p_1,p_0,\alpha,\alpha'$ (note that all the intersection points are nodes except for $p_v$ which is an intersection point of multiplicity $m$) and denote their corresponding  Lefschetz pairs by

 \begin{itemize}
  \item $s_u = [n-3,n-2],\ldots, s_{v+1}=[v+m,v+m+1],$
  \item $s_v = [v+1,v+m]$,
  \item $s_{v-1} = [v,v+1], \ldots, s_0=[1,2], s_{\alpha},s_{\alpha'}$,
\end{itemize}
 \noindent
  where $n-2$ is the number of lines in $\A$ (recall that we get two additional generators from the hyperbola).

Let $u' = x(p')-\varepsilon$ (see Figure \ref{pointMiddleLineCL_hyper3}(d)).
For every singular point $y \in \A$ which is to the left of $u'$, let $s_y$ be the initial skeleton of $y$. We now  consider the braids which are applied on $s_y$ while going along a path $\beta_y$ starting at $x(y)+\varepsilon$ and ending at $u$. This path can be divided into two parts: the first part $\beta_1$ which starts at $x(y) + \varepsilon$ and ends at $u'$ and the second part $\beta_2$ which starts at $u'$ and ends at $u$.

Let $(s'_y)_1$ (resp. $(s'_y)_2$) be the skeleton of $y$ after we apply on it the braids induced by $\beta_1$ (resp. $\beta_y$) in the arrangement $\A'$ and let
$(s_y)_1$ (resp. $(s_y)_2$) be the corresponding skeleton after we apply on it the braids induced by $\beta_1$ (resp. $\beta_y$) in the arrangement $\A$. 

Note that the braid induced by the path $\beta_2$ in $\A'$ is $$\Delta \langle v+1,v+m-1 \rangle,$$ as $p \in \A'$ is the only intersection point
(with multiplicity $m-1$) in the section $\{m \in \C^2 : u' < x(m) < u \}$. Now, the braid induced by the path $\beta_2$ in $\A$ is:
$$
B_2 \doteq \Delta \langle n-3,n-2 \rangle \cdots  \Delta \langle v+1,v+m \rangle \cdots \Delta \langle 2,3 \rangle \Delta \langle 1,2 \rangle  (H(\sigma))^2.
$$

Note that the fibers over $x(y) + \varepsilon$ in $\A$ and  in $\A'$ differ only by one point, as  there is an additional point in the fiber  of $\A$, being the highest point (numbered as $n$) which represents the line $L$. Note also that the skeletons $(s'_y)_1$  and $(s_y)_1$ are the same, as the line $L$ does not intersect the other components in the section between $x(y)$ and $u'$. The only difference is that the fiber over $u'$ in $\A$ has an additional point (being the highest point which represents the line $L$), and the skeleton $(s_y)_1$ does not go around this point.

Now, when applying  $B_2$ on $(s_y)_1$, recall that $\beta_2$ is a path beneath all the intersection points of $L$ (that is, there are no intersection points on $L$ outside the section $\{ a \in \C^2 : u' < x(a) < u \}$). Since  $L$ is almost vertical with a negative slope, the braid $B_2$ induces on the point numbered as $n$  in $\A$ (in the fiber over $u'$) the motion $\delta_{n}$, that is, the point $n$ performs a counterclockwise $180^\circ$ rotation, eventually turning into the point numbered as $1$. Moreover, as the point $n$ (in the fiber over $u'$) is not involved in the skeleton $(s_y)_1$, the only braid that does affect  $(s_y)_1$ is $\Delta \langle v+1,v+m \rangle$. At the stage of applying $\Delta \langle v+1,v+m \rangle$ (in $\A$), the point that was numbered as $n$ is then numbered as $v+m$ and is beneath the skeleton $(s_y)_1$. The braid $ \Delta \langle v+1,v+m \rangle$ sends the point $v+m$ to the point $v+1$ and the effect on the skeleton itself is as the effect of $\Delta \langle v+1,v+m-1 \rangle$. This means that the skeletons $(s'_y)_2$ and $(s_y)_2$ are the same, except for that all the indices in $s''_y$ are increased by $1$.  As was indicated above, we can use the same simplification process with the method described in Part (II) of the proof of Step \ref{lemConjFreeLine} to simplify all the relations.

\medskip

As for the case when the basepoint of $\pi_1(\C^2-\A)$ is $u = x(p')-\varepsilon$, the proof follows the same arguments
as above.

\medskip
\noindent
\textbf{Case (2):} This case covers the remaining two situations: either that $C$ is an ellipse and $p$ is between the two branch points or that
$C$ is a hyperbola and $p$ is not between the two branch points.
 As in case (1), we first rotate the line $L$ to be almost vertical with a negative slope.
 In both situations, the line $L$ intersects $C$ in two \emph{real} points and thus these two situations present the same behavior. After the rotation of $L$, we get the arrangement in Figure \ref{pointMiddleLineCL2}(a) (though Figure \ref{pointMiddleLineCL2} presents the case
 of an ellipse rather than a hyperbola, the proof is independent of the model of $C$ in $\RR^2$, since we do not use the fact that $C$ is an ellipse, but rather the fact that $L$ intersects $C$ in two \emph{real} points).

\begin{figure}[!ht]
\epsfysize 8cm
\epsfbox{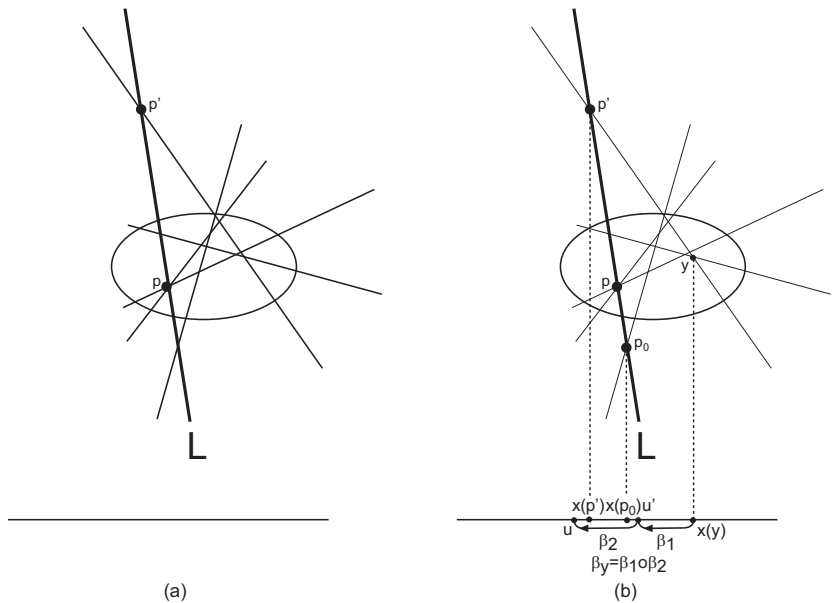}
\caption{Second case: Part (a) is an illustration for the real part of $\A = \A' \cup L$ for  an ellipse. Part (b) depicts notations for the proof of the claim that the relations induced by the singular points to the right of $L$ have no conjugations.}\label{pointMiddleLineCL2}
\end{figure}

Let $p'$ be the highest intersection point of $L$ (with respect to the $y$-coordinate) and
 choose $u \in \ell$ such that $u = x(p')-\varepsilon$, where $0< \varepsilon \ll 1$. We choose the basepoint of $\pi_1(\CC^2 - \A)$ to be $u$.
 Numerate the intersection points on $L$, starting from the lowest point on $L$, by $p_0,\,p_1,\,\ldots,\,p_v=p,\,\ldots,\,p_u=p'$ (note that all the points $p_i$ are nodes except for $p_v$ which is an intersection point of multiplicity $m$) and denote their corresponding  Lefschetz pairs by

 \begin{itemize}
  \item $s_0 = [1,2], s_1 = [2,3], \ldots, s_{v-1} = [v,v+1],$
  \item $s_v = [v+1,v+m]$,
  \item $s_{v+1}=[v+m,v+m+1],\ldots, s_u = [n-1,n]$,
\end{itemize}
\noindent
 where $n-2$ is the number of lines in $\A$ (recall that we get two additional generators from the ellipse). Define  $u' \in \ell$ as $u' = x(p_0)+\varepsilon$.

  We will prove that all the skeletons of the singular points to the left of $u$ or to the right of $u'$ are exactly the same when being computed in $\A$ or in $\A'$. Once we  prove that, combined with the fact that all the relations induced by the singular points on $L$ have no conjugations, then we can use the fact that $\pi_1(\CC^2 - \A',u)$ is conjugation-free and the method described in Part (II) of the proof of Step \ref{lemConjFreeLine} to simplify these relations and to prove that $\pi_1(\CC^2 - \A,u)$
  has a conjugation-free geometric presentation.

  However, the proof now follows exactly the same arguments described in Step \ref{lemConjFreeLine} (indeed, the notations described in Figure \ref{pointMiddleLineCL2}(b) work also for that step), and thus we refer the reader to there.
As for the case when the basepoint of $\pi_1(\C^2-\A)$ is $u = x(p_0)+\varepsilon$, the proof follows the same arguments
as above (see e.g. Step \ref{lemConjFreeLine2}).
\end{proof}

We now prove that moving the basepoint below a projection of a branch point preserves the conjugation-free property.

\begin{step} \label{step2CL}
Let $b$ be a branch point of a conic, and denote $u = x(b)+\varepsilon, u' = x(b)-\varepsilon$. If $\pi_1(\C^2 - \A,u)$ is conjugation-free, then $\pi_1(\C^2 - \A,u')$ is also conjugation-free.
\end{step}

\begin{proof}
Without loss of generality, we can assume that the fiber over $u$ contains $n$ real points, where the points numbered as $c,c+1$ represent the conic, and thus the fiber over $u'$ contains $n-2$ real points and $2$ complex points: indeed, the points numbered as $c,c+1$ perform a $90^\circ$ counterclockwise rotation when passing from $u$ to $u'$. This rotation describes the result of the isomorphism $\pi_1(\C^2 - \A,u) \rightarrow \pi_1(\C^2 - \A,u')$ at the level of the skeletons. However, when we work with a real model $(D,K)$ of the braid group (using  Definition
\ref{defBraidGr}), we can perform any diffeomophism equivalent to the identity and get the same presentation by means of generators and relations. Thus, when taking the basepoint to be $u'$, we can perform on every skeleton a $90^\circ$ clockwise rotation on the points numbered as $c,c+1$. This diffeomorphism is the identity at the level of the generators associated to the conic in $\pi_1(\C^2 - \A,u')$, and therefore we can use the same simplification process as was done in $\pi_1(\C^2 - \A,u)$ for proving that $\pi_1(\C^2 - \A,u')$ is conjugation-free.
\end{proof}

We now prove that $\pi_1(\CC^2 - \A,u)$ has a conjugation-free geometric presentation for every basepoint $u$ that belongs to  one of the following domains:

\begin{itemize}
\item $D_1 = \{ u \in \ell - N : x(p') < u < x(\al) \}$.
\item $D_2 = \{ u \in \ell - N : x(\al) < u \}$.
\item $D_3 = \{ u \in \ell - N : x(p') > u \}$.
\end{itemize}

\begin{step} For $u \in D_1$, $\pi_1(\CC^2 - \A,u)$ has a conjugation-free geometric presentation.  \label{step3CL}
\end{step}

\begin{proof}
The proof is the same as the proof of Step \ref{lemConjFreeLine3}.
\end{proof}

Before proving that $\pi_1(\CC^2 - \A,u)$ has a conjugation-free geometric presentation for the domains $D_2$ and $D_3$, we want
to examine what are the implications that $\pi_1(\CC^2 - \A,u)$ has a conjugation-free geometric presentation for a certain basepoint $u$.
In particular, this means that whenever we have a relation in the presentation of $\pi_1(\CC^2 - \A,u)$, written by the geometric generators which has conjugations of the  generators, these conjugations can be removed.


One important implication of the conjugation-free property is that while the conic induces two geometric generators $x_1, x_2$ in $\pi_1(\C^2 - \A,u)$, the conjugation-free property implies that the relations, coming from \emph{both} branch points, are $x_1=x_2$.
Thus, we can say that not only the conic contributes only one generator in the presentation of $\pi_1(\C^2 - \A,u)$, denoted by $x$, but that there are no
other relations induced by the branch points (apart of $x_1=x_2$). Note that if the presentation is not conjugation-free, we may get
new relations from the branch points which are not $x_1=x_2$ (see e.g. the third step in the proof of \cite[Theorem 4.2]{FG2}).

\medskip

Recall that we assume that $\beta(\A)=0$ and $ \text{deg}(\A)=n$.

\begin{step}\label{prop_commute}
Assume that $\pi_1(\C^2 - \A,u)$ has a conjugation-free geometric presentation.
Let $\G_{k_1},\dots,\G_{k_{n-2}}$ be the geometric generators associated to the $n-2$ lines of $\A$. Then $[x,\G_i] =e$ for all  $1 \leq i \leq n-2$.
\end{step}

\begin{proof}
First, note that if the conic $C$ intersects a line $L_\alpha$ transversally at two nodes, then $[x,\G_\alpha]=e$, due to the conjugation-free property.
Thus we assume that there is at least one multiple point in $\A$ which the conic passes through it.
Note that $\beta(\A)=0$ implies that the graph $G(\A)$ is a forest. We look at the forest $G(\A)$ and we start from a leaf, assuming that the conic passes through the intersection point of multiplicity $m+1$ that corresponds to this leaf (if not, move to its direct ancestor, i.e. to the next step in the proof).

In this case, the  relations induced by this point are:
\begin{equation}
\G_{i_m} \cdots \G_{i_1} x = \G_{i_{m-1}} \cdots \G_{i_1} x \G_{i_m} = \cdots =\G_{i_1} x \G_{i_m} \cdots \G_{i_2} = x \G_{i_m} \cdots \G_{i_1},\label{eqnRel1a}
\end{equation}
where $\G_{i_j}$, $1 \leq j \leq m$, are the geometric generators associated to the lines $L_{i_j}$ which pass through this multiple point.
Note that there are no conjugations in the relations, since the  presentation is a conjugation-free. Since we are dealing with a leaf, all the lines (except maybe for one, which corresponds to the edge connected to its direct ancestor) intersect the conic also in a node. Numerate the lines in such a way that $L_{i_1}$ is the line that possibly does not
intersect the conic in a simple point. Again, since the group has a conjugation-free geometric presentation, we have the following relations, induced by these simple points:

\begin{equation} \label{eqnRel2}
[x,\G_{i_j}]=e,\, \text{ for all } 2 \leq j \leq m.
\end{equation}

Therefore, from relations (\ref{eqnRel1a}) and (\ref{eqnRel2}), one can easily deduce that
$$ [x,\G_{i_1}]=e.$$
Indeed, using relations (\ref{eqnRel2}) and the equality $\G_{i_m} \cdots \G_{i_1} x = x \G_{i_m} \cdots \G_{i_1}$ (left hand side and right hand side of relations (\ref{eqnRel1a})), we have  $[x,\G_{i_1}]=e$ as needed.

\medskip

With this data, we can proceed to the immediate upper level of the tree, which means that we proceed to the second multiple point that is on $L_{i_1} \cap C$ (if it exists).

Now, we do the same process as above to the new level, as this point can now be treated as a ``leaf'', i.e. with the same properties regarding the relations in the fundamental group.
In this way, we go over all the vertices of the graph that the conic passes through the corresponding multiple points.
\end{proof}

\begin{remark} \label{remPropComm}
\emph{Note that the essential data that we have used during the proof of Step \ref{prop_commute} is that there is a vertex (a leaf, to be exact)
in the graph, that all the lines, except maybe for one, that passes through it, intersect the conic in a node. We will use that property in order
to generalize Step \ref{prop_commute} for other arrangements whose graph may contain a cycle.}
\end{remark}

We use Step \ref{prop_commute} in order to prove that if $u$ is in the domains $D_2$ and $D_3$, $\pi_1(\CC^2 - \A,u)$ has a conjugation-free geometric presentation too.

\begin{step} For $u \in D_2= \{ u \in \ell - N : x(\al) < u \}$, $\pi_1(\CC^2 - \A,u)$ has a conjugation-free geometric presentation. \label{step4CL}
\end{step}

\begin{proof}
The proof of this step is based on the following three claims:
\begin{enumerate}
\item  All the skeletons induced by the intersection points not on $L$ are the same in $\A$, as if they were computed in $\A'$,
except for that all the indices in $\A$ are increased by $1$.
\item All the relations induced by the nodes on $L$ have no conjugations.
\item The relation induced by $p \in L$ has no conjugations.
\end{enumerate}

The proofs of claims (1) and (2) are exactly the same as the proofs of the corresponding claims  in Step \ref{lemConjFreeLine4}. However, for claim (3), we have to perform some changes in the proof. As in Step \ref{lemConjFreeLine4}, the proof is by induction on the location of $u$, and we use the same notations from that step. Explicitly, let $q_1,\ldots,q_k$ be the singular points in the domain $D_2$, and we assume that $u = x(q_i)-\varepsilon$, $u' = x(q_i)+\varepsilon$, when we know, by the induction hypothesis that the conjugation-free property holds for
$u \in \{ m \in \ell - N: x(q_1)-\varepsilon < m < x(q_i) \}$.
Recall that the Lefschez pair of the point $p \in L$ (in the fiber above $x(p)+\varepsilon$) is $[v+1,v+m]$.


When the basepoint is $u$, then the relation induced
by the point  $p \in L$ is:
\begin{equation} \label{relMultiPointCL}
[\G_1,\G_{B_u(v+2)},\G_{B_u(v+3)},\ldots,\G_{B_u(v+m)}]=e,
\end{equation}
due to the conjugation-free property of $\pi_1(\CC^2 - \A,u)$.
Note that when going to the basepoint $u'$, the induced isomorphism is $f_3'$ as described in Equation (\ref{eqnf3'}).
However, now we should note that $f_3'$ can affect at most \emph{two} of the generators $\G_{B_u(x)},\, v+2 \leq x \leq v+m$.
This is different from the situation in line arrangements, where  it is impossible that two
\emph{different} lines will pass both through $p$ and   $q_i$. On the other hand, in CL arrangements, the conic and  a line $L'$
can \emph{both} pass through $p$ and $q_i$ (see Figure \ref{twoGen}). Hence, two different  geometric generators, in the relation induced by $p$, might be affected by $f_3'$.

\begin{figure}[!ht]
\epsfysize 5cm
\epsfbox{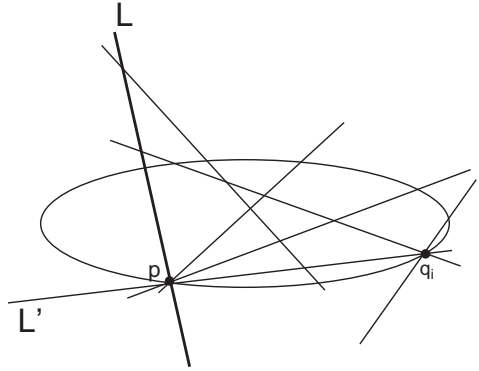}
\caption{The case where the conic and a line $L'$, which pass through $p$, pass  through another singular point $q_i$ too.}\label{twoGen}
\end{figure}

We now divide the proof into three cases: when $f_3'$ does not affect any of the generators $\G_{B_u(x)}$, when it affects only one of them and when it affects two of them. As for the first two cases, the proof is the same as the proof of cases (1) and (2) of claim (3) in the proof of Step \ref{lemConjFreeLine4}. We now deal with the third case.

Recall that the conic contributes two generators to $\pi_1(\C^2-\A,u)$, denoted by $x_1,x_2$. Assume that in the fiber over $u$, the local numeration of the components intersecting in $q_i$ is $a,a+1,\ldots,a+s=b$, where $1<a$ (since the line $L$ is numbered as $1$ in the fiber over $u$). Without loss of generality, assume that over $u$, it is the generator $x_1$ that participates in the relation (\ref{relMultiPointCL}), such that $x_1$ is equal to $\G_{a+t}$, where $0 \leq t \leq s$. Since $\pi_1(\C^2-\A,u)$ is conjugation-free, we know that $x_1=x_2$ in $\pi_1(\C^2-\A,u)$. Moreover, by Step \ref{prop_commute}, we know that $[\G_i,x_2]=e$ for any generator $\G_i$ associated a line in $\A$.

Thus, in $\pi_1(\C^2-\A,u')$,
\begin{equation} \label{eqnf3'x1} f_3'(x_1) = f_3'(\G_{a+t})  \doteq {\Ga' _b} \cdots {\Ga'_{b-t+1}} \Ga_{b-t}' {\Ga'^{-1}_{b-t+1}} \cdots {\Ga _b'}^{-1},\end{equation}
and $f_3'(x_2) = x_2'$. Therefore, the relation $x_1=x_2$ is mapped in $\pi_1(\C^2-\A,u')$ to $f_3'(x_1)=f_3'(x_2)$, which is:
\begin{equation}\label{relEqlConic}
{\Ga' _b} \cdots {\Ga'_{b-t+1}} \Ga_{b-t}' {\Ga'^{-1}_{b-t+1}} \cdots {\Ga _b'}^{-1} = x_2'.
\end{equation}

Now, in $\pi_1(\C^2-\A,u)$, we know that $[\G_a,x_2]=e.$ Thus, in  $\pi_1(\C^2-\A,u')$ we have that
$$
[f_3'(\G_a),f_3'(x_2)]=e  \,\, \rightarrow \,\, [\G_b',x_2']=e.
$$
Also, the existing relation $[\G_{a+1},x_2]=e$  in $\pi_1(\C^2-\A,u)$ is transformed  in $\pi_1(\C^2-\A,u')$ to:
$$
[f_3'(\G_{a+1}),f_3'(x_2)]=e \rightarrow [\G_b'\G_{b-1}'{\G_b'}^{-1},x_2']=e \overset{[\G_b',x_2']=e}{\rightarrow} [\G_{b-1}',x_2']=e.
$$
In the same way, we get that $[\G_j',x_2']=e$ for every $b-t+1 \leq j \leq b$ in $\pi_1(\C^2-\A,u')$. Thus, relation (\ref{relEqlConic}) is simplified to $\G_{b-t}' = x_2'$ and therefore, $e = [\G_j', x_2']=[\G_j',\G_{b-t}']$ for any $b-t+1 \leq j \leq b$. Thus, by Equation (\ref{eqnf3'x1}),
\begin{equation}\label{eqnGammaAt}
f_3'(\G_{a+t}) = \Ga_{b-t}'.
\end{equation}

Now, we deal with the behavior of the generator associated to the line $L'$ passing through $p$ and $q_i$, under the isomorphism $f_3'$.
 We would like to prove that if this generator is $\G_{a+y}$, then $f_3'(\G_{a+y})$ is a conjugation of $\G'_{b-y}$ by other geometric generators which  commute with $\G_1'$.

 We split the treatment into two cases.

If the line $L'$ is numbered  over $u$ as  $a' \doteq a+t-m,$ for $m>0$, then
$$
f_3'(\G_{a+t-m})   \doteq {\Ga' _b} \cdots {\Ga'_{b-t+m+1}} \Ga_{b-t+m}' {\Ga'^{-1}_{b-t+m+1}} \cdots {\Ga _b'}^{-1}.
$$
The conjugating generators of $\Ga_{b-t+m}'$ correspond to lines that intersect the line $L$ in nodes, and thus (by claim (2)) the generator $\G_1'$ commutes with all of them in $\pi_1(\C^2-\A,u')$.

If this line is numbered  over $u$ as  $a' \doteq a+t+m,m>0$, then
$$
f_3'(\G_{a+t+m})   \doteq {\Ga' _b} \cdots {\Ga'_{b-t-m+1}} \Ga_{b-t-m}' {\Ga'^{-1}_{b-t-m+1}} \cdots {\Ga _b'}^{-1}.
$$
Note that one of the conjugating generators of $\Ga_{b-t-m}'$ is $\Ga_{b-t}'$. By Step \ref{prop_commute}, we know that $[\G_1,x_2]=e$ in  $\pi_1(\C^2-\A,u)$.
Thus $$e= [f_3'(\G_1),f_3'(x_2)] = [\G_1',x_2'] \overset{\G_{b-t}'=x_2'}{=} [\G_1',\G_{b-t}']$$ in  $\pi_1(\C^2-\A,u')$. Therefore, using again the same argument above,  $\G_1'$ commutes with all of conjugating elements of $\Ga_{b-t+m}'$.

To conclude, by Equation (\ref{eqnGammaAt}), this means that relation
(\ref{relMultiPointCL}) in  $\pi_1(\C^2-\A,u')$   is turned to
$$
[\G'_1,\G'_{B_u(v+2)},\G'_{B_u(v+3)},\ldots, \G'_{a'-1} , f_3'(\G'_{a'}) ,\G'_{a'+1}, \ldots,\G'_{B_u(v+m)}]=e,
$$
where $f_3'(\G'_{a'})$ equals to a conjugation of a geometric generator with generators, which commute with $\G'_1$. Therefore, we can now proceed exactly as described in the last paragraphs of the proof of Step \ref{lemConjFreeLine4}. \end{proof}

The proof for $u \in D_3$  follows the same arguments as in the case for real line arrangements (see Step \ref{lemConjFreeLine5}),
and therefore, we have proved that $\pi_1(\C^2-\A,u)$ is conjugation-free for any $u \in \ell - N$. \end{proof}

This completes the proof of Proposition \ref{lemAddLineComProof}(2).

%

\section{Results and Applications} \label{secResAndApp}

In the following two subsections, we would like to investigate the cases where $\beta(\A)=0$ or $\beta(\A)=1$ for a CL arrangement $\A$. The first subsection concentrates on the case where $\beta(\A)=0$, examining the structure of the corresponding fundamental group.
The second subsection concentrates on the case where $\beta(\A)=1$, where the conic does not pass through all the multiple points which correspond to the vertices of the graph.

Moreover, we are especially interested to find the cases where the presentation and the structure of the affine fundamental group can be directly read  from the graph (i.e. the lattice of the arrangement determines the fundamental group of its complement), so one of our aims in this section is to find out for which real CL arrangements, the fundamental group is either abelian or conjugation-free.

\subsection{The case of a graph with no cycles} \label{secNoCycles}

In this section, we prove Theorem \ref{main_result}. We first prove that the fundamental groups of the complements of these CL arrangements (with one conic), whose graph has no cycles,
have a conjugation-free geometric presentation (see Proposition \ref{prsBetaZero}), and then we find out the structure of these fundamental groups.

\medskip

Using Proposition \ref{lemAddLineComProof}(2) inductively, we have the following proposition:

\begin{prs}\label{prsBetaZero}
Let $\A$ be a real CL arrangement with one conic and $k$ lines.  Suppose that $\beta (\A)=0$. Then  $\pi_1 (\CC^2 - \mathcal A)$ has a conjugation-free geometric presentation.
\end{prs}

\begin{proof}
First, note that by Lemma \ref{g3_op_CFpreserve_CL}, the conjugation-free geometric presentation is independent of the basepoint for this class of CL arrangements (assuming that the basepoint is real).
Note that $\beta(\A)=0$ implies that the graph $G(\A)$ is a forest. Hence, the CL arrangement can be constructed inductively according to the graph (see an example in Figure \ref{graphBuilding}): first draw the conic and all the lines that do not contribute to the graph (i.e. the lines that do not pass through any multiple point).
Obviously, the fundamental group of this arrangement has a conjugation-free geometric presentation for every real basepoint (since all the components are in general position, the fundamental group of the affine complement is abelian, due to \cite{OkSa}). Now start from the  root of one of the trees, i.e. draw all the lines that correspond to the edges connected to this root (see Figure \ref{graphBuilding}(a)). By Proposition \ref{lemAddLineComProof}(2), the conjugation-free property is preserved. In the following steps, construct the rest of the arrangement by going to the direct successors of the tree's root, and drawing the corresponding lines (see Figures \ref{graphBuilding}(b) and \ref{graphBuilding}(c)). Note that since at each step, we draw a single line passing through only one intersection point and there are no cycles in the resulting graph, the conjugation-free property is preserved. When comparing the resulting arrangement to the original one, the only lines that can be missing are lines that pass only through one multiple point. Thus adding these lines will again preserve the conjugation-free property.

\begin{figure}[!ht]
\epsfysize 3.5cm
\epsfbox{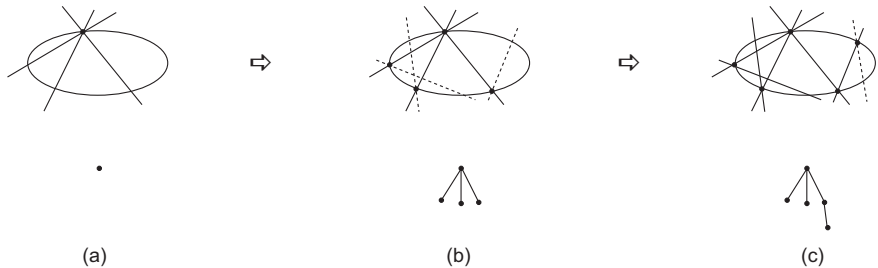}
\caption{An example for an inductive construction of the CL arrangement according to the graph (see Proposition \ref{prsBetaZero}): in step (a), we draw the conic and three more lines which create a multiple point on the conic, which corresponds to the tree's root. In step (b), we add three (dotted) lines creating three new multiple points, thus the associated graph of the arrangement is a tree with a root and three successors. In step (c), we add another (dotted) line creating a new multiple point, which corresponds to a new successor in the tree.}\label{graphBuilding}
\end{figure}

Now, we do the same process to any other tree in the graph, if any. As this process is finite, we are done.
\end{proof}

Now, we can prove Theorem \ref{main_result}.

\begin{proof}[Proof of Theorem \ref{main_result}]
By Step \ref{prop_commute} (which now holds independently of the basepoint $u$), we can conclude that:
$$\pi_1(\C^2 - \A) \cong \langle x \rangle \oplus \pi_1(\C^2 - (\A-C)),$$
where $x$ is the generator of the conic.
Thus, it remains to prove that 
$\pi_1(\C^2 - (\A-C))$ is isomorphic to a direct sum of free groups and a free abelian group. However, this is straight-forward, since  $\beta(\A)=0$ implies that
$\beta(\A-C)=0$. Now, since $\A - C$ is an arrangement of lines, we can use Fan's result \cite{Fa2} that the fundamental group of an arrangement of lines whose graph has no cycles is a direct sum of free groups and a free abelian group.

Explicitly, this means that:
$$\pi_1 (\CC  ^2 - \mathcal A) \cong \ZZ^r \oplus \bigoplus\limits_{i=1}^p \FF_{m(a_i)-2} \oplus \bigoplus\limits_{i=1}^q \FF_{m(b_i)-1},$$
where $r=k+2p+q+1-\sum\limits_{i=1}^p m(a_i)- \sum\limits_{i=1}^q m(b_i)$.

\end{proof}
%

As a result, we have an immediate corollary stating when the fundamental group is abelian.
\begin{corollary}
Let $\A$ be a real CL arrangement with one conic $C$ with only branch points, nodes and triple points as singularities.
Assume that $\beta (\A)=0$ and all the triple points are on the conic. Then  $\pi_1 (\CC^2 - \mathcal A)$ is abelian.
\end{corollary}

\medskip

There are some immediate consequences of Theorem \ref{main_result}, using the following decomposition theorem of Oka and Sakamoto \cite{OkSa}:
\begin{thm}[Oka-Sakamoto]\label{OkaSakamoto}
Let $C_1$ and $C_2$ be algebraic plane curves in $\C ^2$.
Assume that the intersection $C_1 \cap C_2$
consists of distinct $d_1 \cdot  d_2$ points, where $d_i \ (i=1,2)$ are the
respective degrees of $C_1$ and $C_2$. Then:
$$\pi _1 (\C ^2 - (C_1 \cup C_2)) \cong \pi _1 (\C ^2 -C_1) \oplus \pi _1 (\C ^2 -C_2)$$
\end{thm}

We state the consequence for the case of CL arrangements with two conics, but the general case is straight-forward.

\begin{corollary} \label{corCFtwoCon}
Let $\A$ be a real CL arrangement with two conics and $k$ lines.
 Assume that the conics intersect each other transversally and
$\beta (\A)=0$. For $i=1,2$, let
$V_i$ be the set of vertices of $G(\A)$ whose corresponding points lie  on the conic $C_i$. If $G(\A)$ is a disjoint union of two graphs $G_1,G_2$ such that $V_i \subset G_i$, then  $\pi_1 (\CC^2 - \mathcal A)$ is a direct sum of free groups and a free abelian group.
\end{corollary}

\medskip

\begin{remark}
{\rm
The simplest case of a CL arrangement $\A$ with two conics, where we cannot apply Corollary \ref{corCFtwoCon}, is presented in Figure
\ref{exampleNonDisGraph}.
\begin{figure}[!ht]
\epsfysize 4cm
\epsfbox{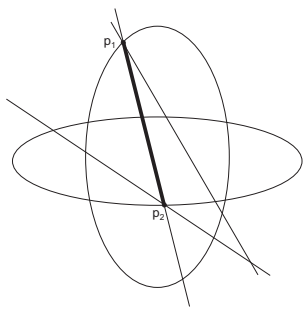}
\caption{The graph of the arrangement $\A$ is a single edge connecting the two vertices corresponding to the triple points $p_1, p_2$.}\label{exampleNonDisGraph}
\end{figure}
However, a direct calculation shows that  $\pi_1 (\CC^2 - \mathcal A)$ has a conjugation-free geometric presentation and thus abelian (e.g. $\pi_1 (\CC^2-\mathcal A)~\cong~\ZZ^5$). Therefore, it is reasonable to conjecture the following:

\begin{conjecture}
Let $\A$ be a real CL arrangement with $n$ conics and $k$ lines, where for each pair of conics, the two conics intersect each other transversally and neither a line nor another conic passes through those intersection points.  Suppose that $\beta (\A)=0$. Then  $\pi_1 (\CC^2 - \mathcal A)$  is a direct sum of free groups and a free abelian group.
\end{conjecture}
}
\end{remark}

We finish this section with the following conjecture for a general smooth plane curve:

\begin{conjecture}
Let $C$ be a smooth plane curve, $\LL$ a real line arrangement, such that for each line $L \in \LL$, $L$ intersects $C$ transversally in a real point.
Define the graph $G(\LL \cup C)$ as in the case of CL arrangements (see Definition \ref{defGraph}). If $\beta(\LL \cup C)=0$, then  $\pi_1(\CC^2 - (\LL \cup C))$  is a direct sum of a free abelian group and free groups.
\end{conjecture}

\subsection{The case of a graph with one cycle, where the conic does not pass through all the vertices of the cycle} \label{secOneCycleNotPass}
In the following section, we want to examine the case where $\beta(\A)~=~1$ (for a CL arrangement with one conic), where the conic does not pass through all the multiple points corresponding to the vertices of the cycle (the case where the conic does pass through all the the multiple points corresponding to the vertices of the cycle is studied in \cite{FG2}).
 We prove that not only that the affine fundamental group of the complement has a conjugation-free geometric presentation, but that the generator of the conic commutes with all the other generators.

\begin{prs} \label{prsCF}
Let $\A$ be a real CL arrangement with one conic $C$ such that:
 \begin{enumerate}
\item $\beta(\A) = 1$.
\item  There is a vertex $y \in G(\A)$ such that $y$ is a vertex contained in the cycle of $G(\A)$, $y \not\in C$ and the two different edges exiting from $y$, which compose the cycle, are associated to two different lines in $\A$ (see Figure \ref{examOneLine}), and both of them intersect the conic also in a node.
\end{enumerate}
Then, $\pi_1 (\CC ^2 - \mathcal A)$ has a conjugation-free geometric presentation.

\end{prs}

\begin{proof}

Let $y \in \A$ be an intersection point satisfying Condition (2). There are $k$ lines $L_1,\dots,L_k$ which pass through $y$; assume that $L_1,L_2$ are the lines which correspond to  edges of the cycle of $G(\A)$ (see Figure \ref{examOneLine}) and $L_3,\dots,L_k$ are not,  otherwise there would be more than one cycle in $G(\A)$. We can also assume that there are no other edges in the graph $G(\A)$ exiting from $y$
(indeed, as will be explained later, by the same methods of  Proposition \ref{prsBetaZero}, adding the corresponding lines will preserve the conjugation-free property).

\begin{figure}[!ht]
\epsfysize 4cm
\epsfbox{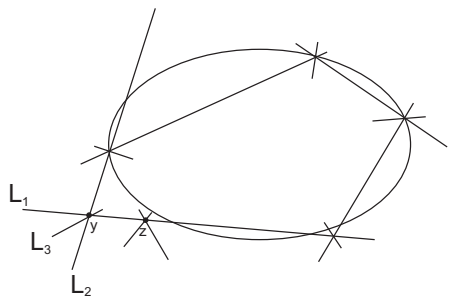}
\caption{The vertex $y$ is contained in the cycle of the graph, and satisfies Condition (2) and $L_1$ and $L_2$ are the lines correspond to  two edges exiting from $y$, intersecting the conic also in a node. The vertex $z$ is contained in the cycle but it does not satisfy Condition (2), since the edges in $G(\A)$ exiting from $z$ are associated to the same line $L_1$.}\label{examOneLine}
\end{figure}

 We look at $\A' = \A - \{L_3,\ldots,L_k\}$. Note that $\beta(\A') = 0$ (since $y$ is a node in $\A'$ and there is only one cycle in $\A$) and therefore $\pi_1 (\CC^2 -  \A')$ has a conjugation-free geometric presentation (by Proposition \ref{prsBetaZero}). Now, draw the line $L_3$ and rotate it till it is
almost vertical to the reference line $\ell$. As was already proven, adding a line in almost vertical position does not affect the conjugation-free property. Now, by the same course of the proof of Lemma \ref{g3_op_CFpreserve_CL}, we can still follow Steps \ref{step1CL}, \ref{step2CL} and \ref{step3CL}. The problem is Step \ref{prop_commute} (proving that all the generators corresponding to the line commute with the generator $x$ corresponding to the conic), as it relies on the fact that $\beta(\A)=0$. However, the proof of Step \ref{prop_commute} starts with choosing one leaf of the graph $G(\A)$ whose corresponding singular point is located on the conic $C$, and proving that all the generators, corresponding to the lines passing through that leaf, commute with $x$ (see Remark \ref{remPropComm}).

But if we choose  the vertex corresponding to $y$ as that ``leaf" here, we see that the same phenomenon happens here; indeed, $L_1$ and $L_2$ intersect $C$  also in a node, and therefore the corresponding generators commute with $x$. $L_3$ intersects $C$ in two nodes, and thus the corresponding generators commute with $x$. This means that for the ancestors of the vertex corresponding to $y$, there is at most only one generator (corresponding to the lines passing through that vertex) that may not commute with $x$ and now we can continue using the same methods of Step \ref{prop_commute} till we cover all the vertices of the cycle of $G(\A)$. Thus  $\pi_1 (\CC^2 -  (\A' \cup L_3))$ has a conjugation-free presentation.  In the same way, we can add the lines $L_4,\ldots,L_k$ inductively and get that
$\pi_1 (\CC^2 - \mathcal A)$ has a conjugation-free presentation.\end{proof}
 The above proof implies that the conic contributes only one generator
to the fundamental group, denoted by $x$ and it commutes with any other generator, corresponding to a line in the arrangement, i.e.:

\begin{prs} \label{lemmaPlusZ}
Let $\A$ be a real CL arrangement with one conic $C$. If $\A$ satisfies the conditions of Proposition \ref{prsCF}, then:
$$\pi_1 (\CC ^2 - \mathcal A) \cong \langle  x \rangle \oplus \pi_1 (\CC ^2 - (\A - C)),$$
where $x$ is the generator associated to the conic.
\end{prs}

\begin{remark}\label{rem_CF_CLarr}
{\rm
(1) Note that if $\A$ is a CL arrangement satisfying the conditions of Proposition \ref{prsCF} and $G(\A - C)$ has no  cycles, then we know that $\pi_1 (\CC^2 - \mathcal (\A-C))$ is a direct sum of free groups and a free abelian group (see \cite{Fa1, Gar}), and thus, by Proposition \ref{lemmaPlusZ}, $\pi_1 (\CC^2 - \mathcal A)$ is also  a direct sum of free groups and a free abelian group.

(2) The class of CL arrangements with a conjugation-free geometric presentation is larger than stated in Proposition \ref{prsCF}. Indeed, it is easy to see that if the graph of a CL arrangement with one conic is a disjoint union of cycles, then the arrangement has a conjugation-free geometric presentation as well.

See also Section \ref{secCFGraphs}, where we define the notion of a {\it conjugation-free graph}, and  show that if the associated graph of the arrangement is a conjugation-free graph, then the fundamental group of the arrangement has a conjugation-free geometric presentation.}
\end{remark}

\begin{example} Based on Remark \ref{rem_CF_CLarr}(1), the affine fundamental group of the complement of the CL arrangement in Figure \ref{graph_com2} is isomorphic to $\Z^2 \oplus \mathbb{F}_2 \oplus \mathbb{F}_2$.

\begin{figure}[!h]
\epsfysize 3cm
\epsfbox{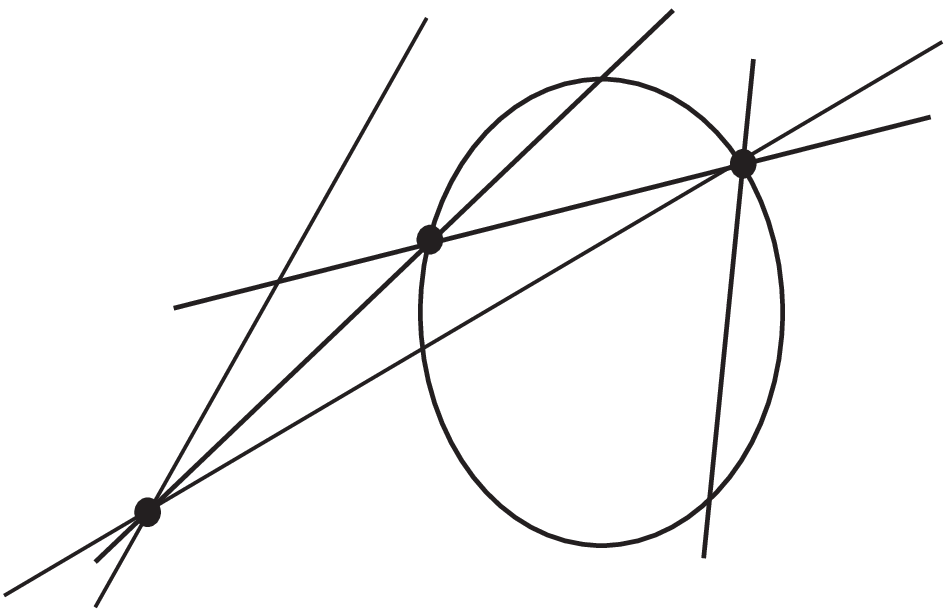}
\caption{}\label{graph_com2}
\end{figure}

\end{example}

\section{Conjugation-free graphs} \label{secCFGraphs}

Following Remark \ref{rem_CF_CLarr}(2) and the methods presented in the previous sections, this section examines a subclass of graphs, associated to a real line arrangement or to a real CL arrangement, which implies that the fundamental group of the arrangement has a conjugation-free geometric presentation.

\subsection[A conjugation-free graph for real line arrangements]{A conjugation-free graph for real line arrangements}\label{subsecCFGline}

We start with the  case of real line arrangements; Recall Proposition \ref{lemAddLineComProof}(1):

\begin{prs} \label{lemmaCF_app}
Let $\LL$ be a real line arrangement such that  $\pi_1(\CC^2 - \LL,u)$ has a conjugation-free geometric presentation for any real basepoint $u \in \ell-N$, where $\ell$ is the reference line. Let $L$ be a real line not in $\LL$ that passes through at most one intersection point of $\LL$. Then  $\pi_1(\CC^2 - (\LL \cup L),u)$ has a conjugation-free geometric presentation for any real basepoint $u$.
\end{prs}

We can use this proposition to reprove the main result of \cite[Proposition 1.4]{EGT1} in an inductive way:
\begin{prs} \label{prsOnlyCycle}
Let $\LL$ be a real line arrangement.
If $G(\LL)$ is a cycle, then $\pi_1(\C^2 - \LL)$ has a conjugation-free geometric presentation.
\end{prs}
\begin{proof}
We build $\LL$ inductively: at each step, we prove that the fundamental group is conjugation-free.

(1) First, draw only the lines that correspond to the edges of the cycle $G(\LL)$: denote by $\LL_0$ the resulting arrangement.
Since $G(\LL)$ is a cycle, then $G(\LL_0)$ is the empty graph, as all the intersection points of $\LL_0$ are  nodes. Thus
$\pi_1(\C^2 - \LL_0)$ is abelian and it is conjugation-free for any basepoint.

(2) Denote by $x_1,\dots,x_s$ the nodes of $\LL_0$ such that their associated vertex appeared in the graph  $G(\LL)$  (i.e. these points correspond to the vertices of $G(\LL)$). Looking at the arrangement $\LL$,  there are $k_i$ lines passing through the point $x_i$: $L_{i_1},\dots, L_{i_{k_i}}$. Note that in $\LL_0$, for each $x_i$ either one or two lines
from the set $\{L_{i_j}\}_{j=1}^{k_i}$ were drawn (see Figure \ref{drawnLines}).

\begin{figure}[!ht]
\epsfysize 3cm
\epsfbox{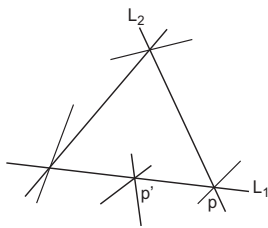}
\caption{For the point $p$, the lines $L_1$ and $L_2$ are drawn; for the point $p'$, only the line $L_1$ is drawn.}\label{drawnLines}
\end{figure}

Look at the point $x_1$ and assume that $L_{1_1}$ and $L_{1_2}$ (or $L_{1_1}$) were already drawn while building $\LL_0$.
Let $\LL_1 \doteq  \{L_{1_1}, L_{1_2}\}$ (or $\LL_1 \doteq \{L_{1_1}\}$ resp.). Then $\LL_1' \doteq \{L_{1_j}\}_{j=1}^{k_i}  - \LL_1$
is a set of lines such that for each $L \in \LL_1'$, $L \cap \{x_i\}_{i=1}^s = x_1$; otherwise (explicitly, if
$L \cap \{x_i\} = \{x_1,x_m\}$, $m \not\in \{2,s\}$) we would get that $\beta(\LL)>1$. Note that if $m=2$ or $m=s$, then $L$ would be already in $\LL_1$ (which is not possible, as $L \in \LL_1'$).

Therefore, each $L \in \LL_1'$ passes through a single intersection point of $\LL_0$ and we can inductively add all the lines
$L \in \LL_1'$ to $\LL_0$ and preserve the conjugation-free property, by Proposition \ref{lemmaCF_app}.

(3) We continue as in step (2) for the points $x_2,\dots,x_s$, adding the missing lines. At each step, the conjugation-free property is preserved.
\end{proof}

Recall the following result, stated in \cite[Corollary 2.5]{EGT2}:

\begin{prs} \label{prsCFless1}
Let $\LL$ be a real line arrangement satisfying $\beta(\LL) \leq 1$. Then, $\pi_1(\CC^2-\LL)$ has a conjugation-free geometric presentation.
\end{prs}
\begin{proof}
Indeed, this was already proved in \cite{EGT2}, and we review the proof shortly. For the case  $\beta(\LL) = 0$, this is the content of Fan's result \cite{Fa2}: in this case, the fundamental group is a direct sum of a free abelian group and free groups and it is easy to see that it is conjugation-free for any basepoint. When
$\beta(\LL) = 1$, we can continue the construction above: all one needs to do is to add the lines corresponding to the trees
whose roots lie on the unique cycle of $G(\LL)$.
\end{proof}

Note that it is shown in \cite{EGT2} that if the graph is a disjoint union of cycles, then the arrangement is also conjugation-free, based on Oka-Sakamoto's theorem (see Theorem \ref{OkaSakamoto}).

\medskip

The above construction motivates the following definition:
\begin{definition} \label{defCFG}  \emph{
Let $G$ be a planar connected graph and denote by deg$(v)$ the number of edges exiting from a vertex  $v \in G$. The graph $G$ is called a \emph{conjugation-free graph} (CFG) if:
\begin{enumerate}
\item $\beta(G) \leq 1$, or
\item Let $\{v_i\}_{i=1}^m$ be the set of vertices in $G$ satisfying deg$(v_i) \leq 2$. For each $v_i$, $1 \leq i \leq m$, denote by $V_i$ the subset of $G$, composed of the vertex $v_i$ and the edge(s) exiting from it. Let $X = X(G) \doteq \bigcup_{i=1}^m V_i$. Then $G$ is a conjugation-free graph if $G - X$ is.
\end{enumerate}
}
\end{definition}

The main result of this section is:

\begin{thm} \label{thmCFG}
Let $\LL$ be a real line arrangement. If $G(\LL)$ is a disjoint union of conjugation-free graphs, then $\pi_1(\CC^2-\LL)$ has a conjugation-free geometric presentation.
\end{thm}

Before proving the theorem, we give some examples of conjugation-free graphs.
\begin{example} \label{exampCFG}
\begin{enumerate}
\item[(1)] Obviously, a forest is a CFG.
\item[(2)] Graphs (a) and (b) in Figure \ref{graphsCF} are CFGs.
\item[(3)] The {\em Ceva arrangement} (also called the {\em braid arrangement}) has a non conjugation-free graph (see Figure \ref{graphsCF}(c)).
\begin{figure}[!ht]
\epsfysize 3cm
\epsfbox{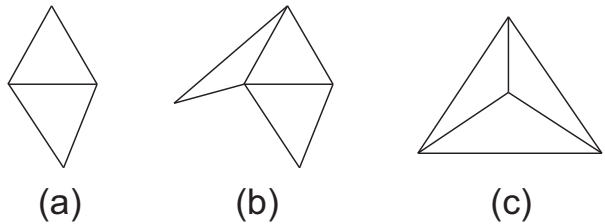}
\caption{Graphs (a) and (b) are conjugation-free graphs. Graph (c) is not a conjugation-free graph.}\label{graphsCF}
\end{figure}
\end{enumerate}
\end{example}

\begin{remark}
Using the package {\it TESTISOM} (see \cite{HR}), one can show that there exist arrangements having a conjugation-free presentation for their fundamental group, though their associated graph is not a conjugation-free graph, e.g. the graph presented in Figure \ref{ceva5}.

\begin{figure}[!ht]
\epsfysize 2cm
\epsfbox{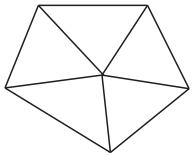}
\caption{A graph which is not a conjugation-free graph, though the fundamental group of its associated line arrangement has a conjugation-free geometric presentation.}\label{ceva5}
\end{figure}

\end{remark}

\begin{proof}[Proof of Theorem \ref{thmCFG}]
We start with the case where $\LL$ is a real line arrangement whose graph $G(\LL)$ is a \emph{connected} conjugation-free graph.
If $\beta(\LL) \leq 1$, then the theorem holds due to Proposition \ref{prsCFless1}.
Assume now that $\beta(\LL) > 1$. Denote $G = G(\LL)$ and $X = X(G)$ (the set of edges and vertices defined in Definition \ref{defCFG}). We split our treatment into two cases.

\medskip

\noindent
{\bf Case (1):}  Assume that the graph $G - X$ has at most one cycle. Let $V = \{v_1, \dots, v_s\} \subset X$ be the set of vertices (of $G$) which were removed. We now look at the  sub-arrangement $\LL_X$ of $\LL$, which is the line arrangement associated to the graph $G - X$; i.e.
$$\LL_X = \LL - \{\text{lines that pass through }v_i,\, 1 \leq i \leq s\}.$$
As $\beta(\LL_X) \leq 1$, the fundamental group of the complement of $\LL_X$ is conjugation-free for any basepoint. Now, for each $v_i \in V \subset X$, there are either one or two lines in $\LL$,
that pass through $v_i$, which correspond to the removed edges in the graph $G$. Now, add these lines to the arrangement $\LL_X$ and call the new arrangement $\LL_X'$. Since these lines pass through at most one multiple point, then $\LL_X'$ is also conjugation-free, by Proposition \ref{lemmaCF_app}.

Note that the every line in $\LL - \LL_X'$ does not correspond to an edge in the graph. Therefore it passes through only one
multiple point of $\LL - \LL_X'$, and thus also $\LL$ is conjugation free.
%

\medskip

\noindent
{\bf Case (2):}
Assume now that the graph $G - X$ has more than one cycle. Let $G_0 = G, X_0 = X, G_1 = G_0 - X_0, X_1 = X(G_1),\dots,G_n = G_{n-1} - X_{n-1}$. Since $G_0$ is a CFG, then there is $k \in \N$  such that the graph $G_k$ has at most one cycle. In this case, we proceed inductively: in the first step, starting from the graph $G_k$, we add to the arrangement corresponding to the graph $G_k$
(which is conjugation-free) the lines that were removed from it in the last step (i.e. the lines that pass through all the vertices with degree $2$ in the graph $G_{k-1}$). As was proved in case (1), we get an arrangement $\LL_{k-1}$ whose fundamental group has a conjugation-free geometric presentation and its graph is $G_{k-1}$. Then we add to $\LL_{k-1}$ its missing lines and we continue till we rebuild the whole arrangement $\LL$.

\medskip

The above two cases finalize the proof for the case of a real line arrangement $\LL$, whose graph $G(\LL)$ is a \emph{connected} conjugation-free graph. The general case of a real line arrangement $\LL$, whose graph $G(\LL)$ is a disjoint union of conjugation-free graphs, can be deduced by Oka-Sakamoto's theorem (see \cite{OkSa} and Theorem \ref{OkaSakamoto} above).
\end{proof}

 We finish this section by
 looking at another operation  $h_r$: rotation of the arrangement $\LL$, when the reference line $\ell$ (and the real basepoint $u$ on it)
remains fixed. We claim that when a real line arrangement  $\LL$ can be built
by adding  one line that passes through at most  one intersection point at each step of its construction (i.e. the graph of $\LL$ is a conjugation-free graph), then $h_r$ preserves the
conjugation-free property.

\begin{lemma} \label{rot_op_CFpreserve}
Let $\LL$ be a real line arrangement such that $G(\LL)$ is a conjugation-free graph.
 Then, on this class of real line arrangements,
the  operation $h_r$ preserves the conjugation-free property for $\LL$; that is, both fundamental groups $\pi_1(\CC^2 - \LL,u)$ and $\pi_1(\CC^2 - h_r(\LL),u)$ are conjugation-free.
\end{lemma}

\begin{proof}
The proof is fairly trivial. We first draw the maximal number of lines of $\LL$ such that its graph is empty, i.e. the resulting
arrangement $\LL_0$ consists only of nodes. Obviously, $\pi_1(\CC^2 - \LL_0,u)$ is conjugation-free, as it is abelian, for every basepoint $u$.
It is also clear that for $h_r(\LL_0)$ the rotated arrangement, $\pi_1(\CC^2 - h_r(\LL_0),u)$ is conjugation-free for every basepoint $u$.
Now, for both arrangements, we can draw one line at each step of the construction, passing through at most  one intersection point.
By Lemma \ref{g3_op_CFpreserve}, at each step we still get that the two resulting arrangements have fundamental groups which are
conjugation-free, and thus, eventually after
drawing all the lines (in $\LL - \LL_0$ and in $h_r(\LL) - h_r(\LL_0)$), we get that both fundamental groups have the conjugation-free property.
\end{proof}

\begin{remark}
\emph{While other lemmata (such as Lemma \ref{g1_op_CFpreserve} and  Lemma \ref{g2_op_CFpreserve}) prove the preservation of the conjugation-free property
by finding the explicit isomorphism, the above lemma does not have to find it (even though the presentations of both fundamental groups are different), as we are using only the property that adding a line (that pass through at most  one intersection point) preserves the conjugation-free property.}
\end{remark}

\subsection[A conjugation-free graph for real CL arrangements]{A conjugation-free graph for real CL arrangements}
 In this section, we check which results from  Section \ref{subsecCFGline} can be adapted to the case of real CL arrangements.

 Note that if $\A$ is a real CL arrangement with only one conic, then the definition of a conjugation-free graph can be applied
 with some changes. Explicitly:

 \begin{definition} \label{defCFG}  \emph{
Let $G$ be a planar connected graph  and let $\A$ be a real CL arrangement with one conic $C$. The graph $G$ is called a \emph{conjugation-free graph} (CFG) associated to $\A$ if $G = G(\A)$ and:
\begin{enumerate}
\item $\beta(G) = 0$, or $\beta(G) = 1$ and $G$ satisfies  Condition (2) of Proposition~\ref{prsCF},\\ or
\item Let $\{v_i\}_{i=1}^m$ be the set of all vertices in $G$ satisfying deg$(v_i) \leq 2$ and the point corresponds to the vertex $v_i$ is not on $C$. For each $v_i$, $1 \leq i \leq m$, denote by $V_i$ the subset of $G$, composed of the vertex $v_i$ and the edge(s) exiting from it. Let $X = X(G) \doteq \bigcup_{i=1}^m V_i$. Then $G$ is a conjugation-free graph if $G - X$ is.
\end{enumerate}
}
\end{definition}

  Indeed, assume that we are given a real CL arrangement $\A$ with only one conic $C$ and the graph $G(\A)$ is a conjugation-free graph.
   Given an arrangement $\A'$, such that $C \in \A'$, with a graph $G' \subset G$ with $\beta(G') = 0$, or $\beta(G') = 1$ and $G'$ satisfies the conditions of Proposition \ref{prsCF}(2), we can add the ``missing" lines (i.e. all the lines in $\A - \A'$) to this arrangement, in the same way described in the proof of Theorem \ref{thmCFG}, and the  conjugation-free property of this arrangement is preserved (since all the lines intersect the conic in nodes, as we pass them through points not on $C$, we can use the same arguments that were used in the proof of Proposition \ref{prsCF}).

Note that adding a conic $C_2$, passing through a single intersection point, to a CL arrangement $\A$ with one conic $C_1$, which has a conjugation-free graph associated to it, also preserves
the conjugation-free property, under the following condition: Let $C_1,C_2$ be the two conics of the CL arrangement $\A \cup C_2$. For $i=1,2$, let $V_i$ be the set of vertices of $G(\A \cup C_2)$ which lie  on the conic $C_i$. If $G(\A \cup C_2)$ is a disjoint union of two conjugation-free graphs $G_1,G_2$ such that $V_i \subset G_i$, then $\pi_1 (\CC^2 - (\A \cup C_2))$ has a conjugation-free geometric presentation (by Corollary \ref{corCFtwoCon}).

\medskip

We finish this section with the following conjecture, related to real line arrangements with a fundamental group having a conjugation-free geometric presentation:
\begin{conjecture}
Let $\LL$ be a real line arrangement such that  $\pi_1(\CC^2 - \LL)$ has a conjugation-free geometric presentation, whose associated graph might not be a conjugation-free graph. Let $C$ be a conic that passes through a single intersection point of $\LL$. Then  $\pi_1(\CC^2 - (\LL \cup C))$ has a conjugation-free geometric presentation as well.
\end{conjecture}

\end{document}